%% file: main.tex
\documentclass{article}

\usepackage[T1]{fontenc}
\usepackage[utf8]{inputenc}

\usepackage{subcaption}

\usepackage{adjustbox}

\usepackage{enumitem}

\usepackage{dsfont} % \mathds 1, a blackboard 1
\usepackage{amssymb}

\usepackage{amsmath}
\usepackage{amsthm}
\newtheorem{theorem}{Theorem}
\newtheorem{lemma}[theorem]{Lemma}
\newtheorem{corollary}[theorem]{Corollary}

\newtheorem{conjecture}[theorem]{Conjecture}
\newtheorem{question}[theorem]{Question}

\usepackage[nocompress]{cite}

\usepackage{tikz}
\usetikzlibrary{
    calc,
    positioning,
    backgrounds,
}
\usepackage{pgfplots}
\pgfplotsset{compat=1.14}

\pgfplotsset{
    closed dot/.style = {
        only marks,
        mark = *,
        mark size = 1.5pt,
    },
    open dot/.style = {
        only marks,
        mark = *,
        fill = white,
        mark size = 1.5pt,
    },
}

\newcommand{\floor}[1]{\left\lfloor #1 \right\rfloor}
\newcommand{\ceil}[1]{\left\lceil #1 \right\rceil}
\newcommand{\norm}[1]{\left\lVert #1 \right\rVert}
\DeclareMathOperator{\aff}{aff} % Affine hull
\DeclareMathOperator{\conv}{conv} % Convex hull
\DeclareMathOperator{\vol}{vol} % Volume
\DeclareMathOperator{\ppyr}{ppyr} % Pseudo Pyramid
\newcommand{\rvol}{\operatorname{vol_r}} % Relative Volume

\usepackage[hidelinks]{hyperref}

\makeatletter
\newcommand{\restate@theorem@name}{}
\newtheorem*{restate@theorem}{\restate@theorem@name}
\newenvironment{restatetheorem}[1]{
    \renewcommand{\restate@theorem@name}{Theorem~\ref{#1}}
    \begin{restate@theorem}
}{
    \end{restate@theorem}
}

\begin{document}

%\title{On the non-translation invariance of the real-parameter Ehrhart function}
\title{
    Reconstruction of rational polytopes
    from the real-parameter Ehrhart function of its translates
}
\author{
    Tiago Royer
}
\date{}
\maketitle

\begin{abstract}
    When extending the Ehrhart lattice point enumerator $L_P(t)$
    to allow real dilation parameters $t$,
    we lose the invariance under integer translations
    that exists when $t$ is restricted to be an integer.
    This paper studies this phenomenon;
    in particular,
    it is shown that,
    for full-dimensional $P$,
    not only there are infinitely many different functions $L_{P + w}(t)$
    (for integer $w$),
    but that for rational $P$ the collection of these functions identifies $P$ uniquely.
\end{abstract}

\input{introduction}
\input{non-rational-polytopes}
\input{notation}

\input{translation-variant}

\input{semirational-polytopes}
\input{discontinuities-of-ehrhart-function}
\input{relative-volumes-of-facets}
\input{isolating-largest-vector}
\input{pseudodiophantine-equations}
\input{piecing-together-rational-case}

\input{codimension-one}
\input{avoiding-overlapping-discontinuities}
\input{piecing-together-avoiding-discontinuities}
\input{semirational-codimension-one-polytopes}
\input{all-the-way-down}

\input{final-remarks}

\bibliographystyle{plainurl}
\bibliography{bib}

\end{document}

%% file: introduction.tex
\section{Introduction}

Given a polytope $P \subseteq \mathbb R^d$,
the classical Ehrhart lattice point enumerator $L_P(t)$ is defined by
\begin{equation*}
    L_P(t) = \#(tP \cap \mathbb Z^d), \qquad \text{integer $t \geq 0$.}
\end{equation*}
Here,
$\#(A)$ is the number of elements in $A$
and $tP = \{tx \mid x \in P\}$ is the dilation of $P$ by $t$.
This function is well-studied
(see~\cite{ccd}, for example),
and more recently some papers have studied the extension
where $t$ is allowed to be any real number~%
\cite{Linke2011, BBKV2013, HenkLinke2015, BBLKV2016, Borda2016}.

To minimize confusion,
we will denote real dilation parameters with the letter $s$,
so that $L_P(t)$ denotes the classical Ehrhart function
and $L_P(s)$ denotes the extension considered in this paper.
Thus,
$L_P(t)$ is just the restriction of $L_P(s)$ to integer arguments.

It is clear from the definition that
the classical Ehrhart function is invariant under integer translations;
that is,
for every real polytope $P$ and every integer vector $w$,
we have
\begin{equation*}
    L_{P + w}(t) = L_P(t)
\end{equation*}
for all integer $t$.
This is not true for the real Ehrhart function $L_P(s)$.
If $P$ contains the origin,
it is easy to see that $L_P(s)$ is a nondecreasing function.
It is an easy exercise to show that the one-sided limit
\begin{equation*}
    L_P(0^+) = \lim_{s \to 0^+} L_P(s)
\end{equation*}
is zero if $P$ does not contain the origin.
If we agree that $L_P(0) = 1$ for all $P$,
then the function $L_P(s)$ will not be nondecreasing in this case.
Therefore,
if $P$ is,
for example,
an integer polytope which does not contain the origin,
and $v$ is any vertex of $P$,
then $L_{P - v}(s)$ is be nondecreasing
whereas $L_P(s)$ is not,
and thus surely these two functions are different.

In Section~\ref{sec:translation-variant},
this result is strengthened to the following.

\begin{theorem}
    \label{thm:rational-translation-variant}
    Let $P$ be a rational polytope.
    Then there exists an integer vector $w$
    such that the functions $L_{P + kw}(s)$ are all distinct
    for $k \geq 0$.
\end{theorem}

This is in sharp contrast with the classical Ehrhart function,
where all the functions $L_{P + k w}(t)$ are the same.

Since $L_P(s)$ is nondecreasing if and only if $P$ contains the origin,
we have that $L_{P - v}(s)$ is nondecreasing
if and only if $v \in P$.
Therefore,
if we know the real Ehrhart functions of all real translates of $P$,
then we can reconstruct the polytope $P$.

The main result of this paper is the following.

\begin{theorem}
    \label{thm:rational-complete-invariant}
    Let $P$ and $Q$ be two rational polytopes such that,
    for every integer translation vector $w$,
    we have $L_{P + w}(s) = L_{Q + w}(s)$.
    Then $P = Q$.
\end{theorem}

In other words,
if $P$ is a rational polytope,
then the Ehrhart functions $L_{P + w}(s)$ (for integer $w$) indentifies $P$ uniquely,
so we may reconstruct $P$ from the Ehrhart function of its integer translates.

We may also see this theorem as a first step
towards what Fernandes, Pina, Ramírez-Alfonsín and Robins call
the ``Hilbert's third problem for the unimodular group'' \cite{Cris2017}.
The conjecture is that,
if two polytopes have the same Ehrhart function,
then they must be piecewise unimodular images of each other.
Theorem~\ref{thm:rational-complete-invariant} shows that,
given more information,
the conclusion is, indeed, true.
In fact,
with the information provided,
the theorem concludes that the polytopes are actually the same
(not even up to translation).
This suggests that we may drop some of that information
and still conclude that the polytopes are piecewise unimodular images of each other.

%% file: non-rational-polytopes.tex
\subsection{Non-rational polytopes}
\label{sec:non-rational-polytopes}

Although Ehrhart theory is usually concerned with rational polytopes,
there has been some effort in working with non-rational polytopes as well.
For example,
Borda~\cite{Borda2016} deals with simplices and cross-polytopes
with algebraic coordinates;
and \cite{DeSarioRobins2011} and~\cite{DiazLeRobins2016}
deal with arbitrary real polytopes and real dilations
but for the solid-angle polynomial.

Theorem~\ref{thm:rational-translation-variant}
may be extended to all real polytopes;
we just need to assume that the polytope is either full-dimensional
or has codimension $1$.
That is,
we have the following.

\begin{theorem}
    \label{thm:translation-variant}
    Let $P \subseteq \mathbb R^d$ be a real polytope
    which is either full-dimensional
    or has codimension $1$.
    Then there is an integral vector $w \subseteq \mathbb R^d$ such that
    the functions $L_{P + k w}(s)$ are different
    for all integers $k \geq 0$.
\end{theorem}

The dimensionality assumption is indeed necessary:
if $P$ is any polytope contained in the affine space
$\{(\ln 2, \ln 3)\} \times \mathbb R^{d-2}$,
then $s(P + w)$ will never contain integer points,
for all integer $w$ and all real $s > 0$.

The extension of Theorem~\ref{thm:rational-complete-invariant}
is more modest.
Write a polytope $P \subseteq \mathbb R^d$ as
\begin{equation*}
    P = \bigcap_{i = 1}^n \{ x \in \mathbb R^d \mid \langle a_i, x \rangle \leq b_i \}.
\end{equation*}
If all $a_i$ are integral vectors,
with the $b_i$ being arbitrary real numbers,
then $P$ is called \emph{semi-rational}.
Semi-rational polytopes seem natural in this context
because,
if $P$ is a semi-rational polytope,
then for any vector $v$ and any real $s > 0$
the polytopes $P + v$ and $sP$ are semi-rational.
For example,
$P = [0, \sqrt 2] \times [0, \sqrt 3]$ is a semi-rational polytope,
although it is not a translation or a dilation of a rational polytope.

We have the following generalization.

\begin{theorem}
    \label{thm:codimension-zero-and-one}
    Let $P$ and $Q$ be two semi-rational polytopes in $\mathbb R^d$,
    both having codimension $0$ or $1$.
    Suppose moreover that $L_{P + w}(s) = L_{Q + w}(s)$
    for all integer $w$ and all real $s > 0$.
    Then $P = Q$.
\end{theorem}

%% file: notation.tex
\subsection{Notation and structure of the paper}

We will usually represent polytopes
by their description as intersection of half-spaces.
That is,
we will write polytopes $P \subset \mathbb R^d$ as
\begin{equation*}
    P = \bigcap_{i = 1}^n \{x \in \mathbb R^d \mid \langle a_i, x \rangle \leq b_i \},
\end{equation*}
where $a_i$ are vectors of $\mathbb R^d$ and $b_i$ are real numbers.

Section~\ref{sec:translation-variant}
contains the proof of Theorems
\ref{thm:translation-variant} and~\ref{thm:rational-translation-variant}.
Most of the time,
we will be working with arbitrary polytopes on this section,
so the vectors $a_i$ will be assumed to be normalized.

Sections \ref{sec:semirational-polytopes} and~\ref{sec:codimension-one}
deal with semi-rational polytopes,
so in this section the $a_i$ will be primitive integer vectors;
that is,
vectors $a_i$ such that $\frac 1 k a_i$ is not an integer
for all integer $k > 1$,
or, equivalently,
the greatest common divisor of all coordinates of $a_i$ is $1$.

In Section~\ref{sec:semirational-polytopes},
we will show Theorem~\ref{thm:codimension-zero-and-one}
just for full-dimensional semi-rational polytopes;
this is Corollary~\ref{thm:semirational-complete-invariant}.
It turns out that showing this theorem for codimension one polytopes
is actually harder than for full-dimensional ones;
in fact,
in Section~\ref{sec:piecing-together-avoiding-discontinuities},
we will arrive at Corollary~\ref{thm:semirational-dense-information-complete-invariant}
which is a strengthened version of Corollary~\ref{thm:semirational-complete-invariant}
that says that,
for full-dimensional semi-rational polytopes $P$ and $Q$,
we have $P = Q$ even if $L_{P + w}(s) = L_{Q + w}(s)$
for $s$ in a dense subset of $\mathbb R$.
Then we will reduce codimension one semi-rational polytopes,
and rational polytopes with any dimension,
to this corollary.

If $D \subseteq \mathbb R$ is an unbounded set,
then the expression
\begin{equation*}
    \lim_{\substack{
        s \to \infty \\
        s \in D
    }}
    f(s)
    = l
\end{equation*}
means that,
for every $\varepsilon > 0$,
there is a number $N$ such that,
for all $s \in D$,
if $s > N$ then $|f(s) - l| < \varepsilon$.
In other words,
this is a ``limit with restricted domain''.
For example,
if $D = \{k \pi \mid k \in \mathbb Z\}$,
then
\begin{equation*}
    \lim_{\substack{
        \theta \to \infty \\
        \theta \in D
    }}
    \sin \theta
    = 0,
\end{equation*}
a limit which we will usually write as
\begin{equation*}
    \lim_{\substack{
        \theta \to \infty \\
        \frac\theta\pi \in \mathbb Z
    }}
    \sin \theta = 0.
\end{equation*}

Note that this limit is undefined if the domain is bounded.

We will also use the Iverson bracket,
which are defined as follows.
Given a proposition $p$,
we define the number $[p]$ to be $1$ if $p$ is true,
and $0$ if $p$ is false.
For example,
the number $\#(P \cap \mathbb Z^d)$ of integer points contained in $P$
may be expressed as
\begin{equation*}
    \#(P \cap \mathbb Z^d) = \sum_{x \in \mathbb Z^d} [x \in P].
\end{equation*}

The relative volume of a semi-rational polytope $P$,
denoted by $\rvol P$,
is defined in Section~\ref{sec:relative-volumes-of-facets}.

%% file: translation-variant.tex
\section{Pseudopyramids and non-translation invariance}
\label{sec:translation-variant}

In this section,
it will be shown that the real Ehrhart function $L_P(s)$
is ``very far'' from being translation invariant.
We will first define an operation,
called ``pseudopyramid'',
that constructs a polytope $\ppyr(P)$ from a polytope $P$.
We will show that it is possible to reconstruct $L_{\ppyr P}(s)$ from $L_P(s)$.
Then we will show that,
for appropriate integer $w$,
the polytopes $\ppyr P$ and $\ppyr(P + w)$ will have different volumes;
this means that $L_{\ppyr P}(s)$ and $L_{\ppyr(P + w)}(s)$ are different,
which implies that $L_P(s)$ and $L_{P + w}(s)$ differ.

Let $P \subseteq \mathbb R^d$ be any polytope.
Define the pseudopyramid $\ppyr(P)$ of $P$
to be the convex hull of $P \cup \{0\}$,
or,
equivalently,
\begin{equation*}
    \ppyr(P) = \bigcup_{0 \leq \lambda \leq 1} \lambda P
\end{equation*}
(Figure~\ref{fig:pseudopyramid}).
The pseudopyramid is so called
because it resembles the operation of creating a pyramid over a polytope.
Note however that the pseudopyramid lives in the same ambient space as the polytope,
whereas the pyramid over a polytope is a polytope in one higher dimension
(that is, $\ppyr(P) \subseteq \mathbb R^d$
while $\operatorname{pyr}(P) \subseteq \mathbb R^{d+1}$).

\begin{figure}[t]
    \centering
    \begin{tikzpicture}
        \draw (-1, 0) -- (3, 0);
        \draw (0, -1) -- (0, 3); % Axis
        \filldraw [thick, blue, fill = blue!40] (1, 1) rectangle (2.5, 2.5);
        \node at (1.75, 1.75) {$P$};

        \begin{scope}[xshift = 5cm]
            \draw (-1, 0) -- (3, 0);
            \draw (0, -1) -- (0, 3); % Axis
            \filldraw [thick, blue, fill = blue!40]
                (0, 0) -- (1, 2.5) -- (2.5, 2.5) -- (2.5, 1) -- cycle;
            \node at (1.75, 1.75) {$\ppyr(P)$};
        \end{scope}
    \end{tikzpicture}
    \caption{
        Pseudopyramid of a polytope.
    }
    \label{fig:pseudopyramid}
\end{figure}
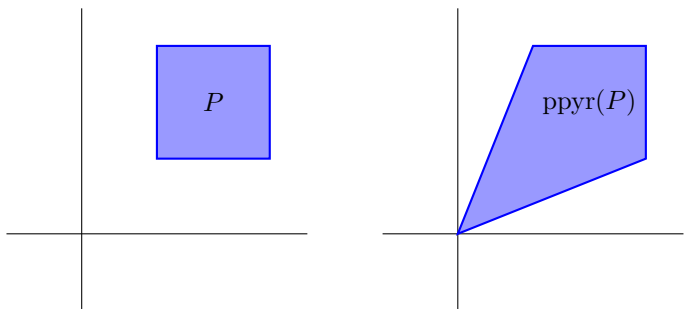

If $x \in \lambda P$ for some $\lambda \in [0, s]$,
then $x \in s \ppyr(P)$.
Therefore,
if $s \ppyr(P)$ and $s \ppyr(Q)$ contain a different number of points,
we can be sure $L_P(s)$ and $L_Q(s)$ differ somewhere in the interval $[0, s]$.
More precisely:

\begin{lemma}
    \label{thm:different-pseudopyramid-volumes}
    Let $P$ and $Q$ be real polytopes such that $L_P(s) = L_Q(s)$.
    Then $L_{\ppyr P}(s) = L_{\ppyr Q}(s)$.
\end{lemma}

\begin{proof}
    First,
    we will define an operation,
    called ``lifting'',
    which we will use to reconstruct $L_{\ppyr P}(s)$ from $L_P(s)$.

    Let $f: \mathbb R \to \mathbb R$
    be any function which has a jump-discontinuity at a point $s_0$,
    and denote by $f(s_0^+)$ the limit of $f(s)$ as $s \to s_0$ with $s > s_0$.
    Define a function $g: \mathbb R \to \mathbb R$ by
    \begin{equation*}
        g(s) = \begin{cases}
            f(s), & \text{if $s \leq s_0$;} \\
            f(s) - f(s_0^+) + f(s_0), & \text{if $s > s_0$.}
        \end{cases}
    \end{equation*}
    The function $g$ is right-continuous at $s_0$ by construction.
    Call $g$ the \emph{result of lifting $f$ at $s_0$}.
    For example,
    if $f$ is the indicator function of $[0, 1]$,
    the result of lifting $f$ at $1$ is the indicator function of $[0, \infty)$.

    If the discontinuity points of $f$ are $s_0 < s_1 < s_2 < \dots$,
    we may successively lift the function at these points;
    that is,
    let $f_0 = f$
    and for $k \geq 0$ let $f_{k+1}$ be the result of lifting $f_k$ at $s_k$.
    If $s < s_n$ for some $n$,
    then for all $k > n$ we have $f_k(s) = f_n(s)$,
    so that the functions $f_k$ converge pointwise at every $s \in \mathbb R$.
    Let $g$ be this pointwise limit;
    we will call $g$ the \emph{lifting} of $f$.
    (Figure~\ref{fig:ehrhart-function-of-ppyr}
    shows the graph of the lifting of the function
    depicted in Figure~\ref{fig:ehrhart-function-of-square}.)

    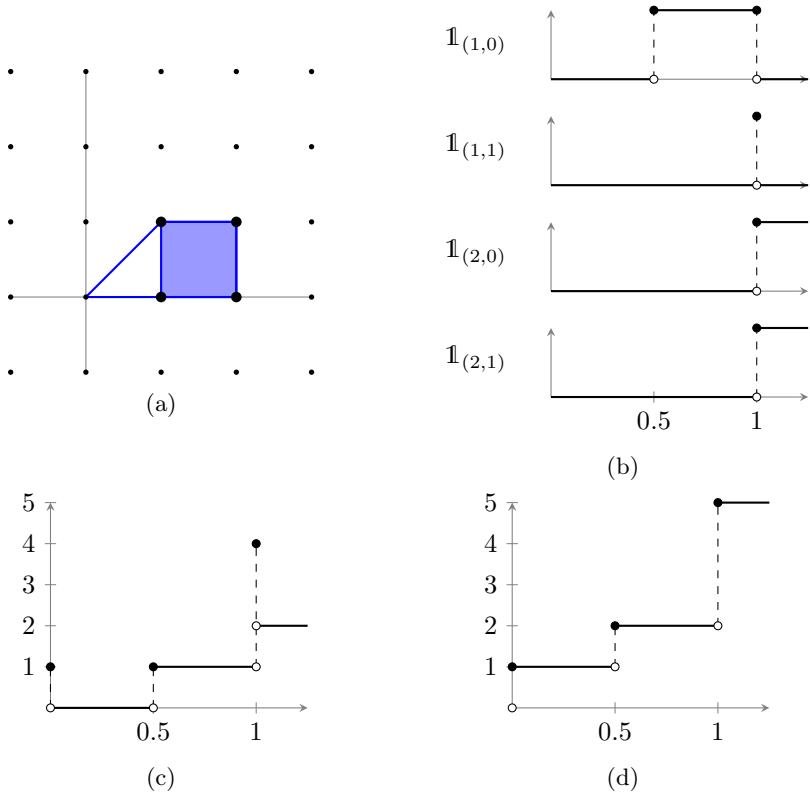
\begin{figure}[t]
        \pgfplotsset{
            every axis/.style = {
                width = 5cm,
                axis lines = center,
                clip = false,
                xmin = 0, xmax = 1.25,
            },
            axis line style = {gray},
        }
        \begin{subfigure}{.5\linewidth}
            \centering
            \begin{tikzpicture}
                \draw [gray] (-1, 0) -- (3, 0);
                \draw [gray] (0, -1) -- (0, 3);

                \filldraw [blue, fill = blue!40, thick]
                    (1, 1) -- (2, 1) -- (2, 0) -- (1, 0) -- cycle;

                \draw [blue, thick] % The pseudopyramid
                    (0, 0) -- (1, 1) -- (2, 1) -- (2, 0) -- (1, 0) -- cycle;

                \foreach \x in {-1, ..., 3}
                    \foreach \y in {-1, ..., 3}
                        \fill (\x, \y) circle [radius = 1pt];

                \foreach \point in {(1, 1), (2, 1), (2, 0), (1, 0)}
                    \fill \point circle [radius = 2pt];
            \end{tikzpicture}
            \caption{}
            \label{fig:ppyr-of-square}
        \end{subfigure}
        \begin{subfigure}{.5\linewidth}
            \centering
            \pgfplotsset{
                every axis/.style = {
                    xtick = \empty,
                    ytick = \empty,
                    width = 5cm,
                    height = 2.5cm,
                    axis lines = center,
                    clip = false,
                    xmin = 0, xmax = 1.25,
                    ymin = 0, ymax = 1,
                },
            }

            \begin{tikzpicture}
                \node at (-1, 0.5) {$\mathds 1_{(1, 0)}$};
                \begin{axis}
                    \draw [thick] (0, 0) -- (0.5, 0)
                        (0.5, 1) -- (1, 1)
                        (1, 0) -- (1.25, 0);
                    \draw [dashed] (0.5, 0) -- (0.5, 1)
                        (1, 0) -- (1, 1);
                    \addplot [closed dot] coordinates {(0.5, 1) (1, 1)};
                    \addplot [open dot] coordinates {(0.5, 0) (1, 0)};
                \end{axis}
            \end{tikzpicture}
            \\[1em]

            \begin{tikzpicture}
                \node at (-1, 0.5) {$\mathds 1_{(1, 1)}$};
                \begin{axis}
                    \draw [thick] (0, 0) -- (1.25, 0);
                    \draw [dashed] (1, 0) -- (1, 1);
                    \addplot [closed dot] coordinates {(1, 1)};
                    \addplot [open dot] coordinates {(1, 0)};
                \end{axis}
            \end{tikzpicture}
            \\[1em]

            \begin{tikzpicture}
                \node at (-1, 0.5) {$\mathds 1_{(2, 0)}$};
                \begin{axis}
                    \draw [thick] (0, 0) -- (1, 0)
                        (1, 1) -- (1.25, 1);
                    \draw [dashed] (1, 0) -- (1, 1);
                    \addplot [closed dot] coordinates {(1, 1)};
                    \addplot [open dot] coordinates {(1, 0)};
                \end{axis}
            \end{tikzpicture}
            \\[1em]

            \begin{tikzpicture}
                \node at (-1, 0.5) {$\mathds 1_{(2, 1)}$};
                \begin{axis}[xtick = {0.5, 1}]
                    \draw [thick] (0, 0) -- (1, 0)
                        (1, 1) -- (1.25, 1);
                    \draw [dashed] (1, 0) -- (1, 1);
                    \addplot [closed dot] coordinates {(1, 1)};
                    \addplot [open dot] coordinates {(1, 0)};
                \end{axis}
            \end{tikzpicture}

            \caption{}
            \label{fig:indicator-function-of-points}
        \end{subfigure}

        \begin{subfigure}{.5\linewidth} % L_P
            \centering
            \begin{tikzpicture}
                \begin{axis} [
                    ytick = {0, ..., 5},
                    ymin = 0, ymax = 5, % To synchronize with the graph below
                ]
                    \addplot [closed dot] coordinates
                        {(0, 1) (0.5, 1) (1, 4)};
                    \addplot [open dot] coordinates
                        {(0, 0) (0.5, 0) (1, 1) (1, 2)};
                    \draw [thick] (0, 0) -- (0.5, 0)
                        (0.5, 1) -- (1, 1)
                        (1, 2) -- (1.25, 2);
                    \draw [dashed] (0, 0) -- (0, 1)
                        (0.5, 0) -- (0.5, 1)
                        (1, 1) -- (1, 4);
                \end{axis}
            \end{tikzpicture}

            \caption{}
            \label{fig:ehrhart-function-of-square}
        \end{subfigure}
        \begin{subfigure}{.5\linewidth} % L_ppyr P
            \centering
            \begin{tikzpicture}
                \begin{axis} [
                    ytick = {0, ..., 5},
                    ymin = 0, ymax = 5, % To synchronize with the graph above
                ]
                    \addplot [closed dot] coordinates
                        {(0, 1) (0.5, 2) (1, 5)};
                    \addplot [open dot] coordinates
                        {(0, 0) (0.5, 1) (1, 2)};
                    \draw [thick] (0, 1) -- (0.5, 1)
                        (0.5, 2) -- (1, 2)
                        (1, 5) -- (1.25, 5);
                    \draw [dashed] (0.5, 1) -- (0.5, 2)
                        (1, 2) -- (1, 5);
                \end{axis}
            \end{tikzpicture}

            \caption{}
            \label{fig:ehrhart-function-of-ppyr}
        \end{subfigure}
        \caption[
            Computing $L_{\ppyr P}(s)$ from $L_P(s)$.
        ] {
            Computing $L_{\ppyr P}(s)$ from $L_P(s)$.
            (\subref{fig:ppyr-of-square}): Polytope $P = [1, 2] \times [0, 1]$,
            and an outline of its pseudopyramid.
            (\subref{fig:indicator-function-of-points}): ``Indicator functions''
            $\mathds 1_x(s)$ of the points $(1, 0)$, $(1, 1)$, $(2, 0)$ and $(2, 1)$.
            (\subref{fig:ehrhart-function-of-square}): Function $L_P(s)$.
            (\subref{fig:ehrhart-function-of-ppyr}): Function $L_{\ppyr P}(s)$.
        }
        \label{fig:l-ppyr-from-l-p}
    \end{figure}

    Fix the polytope $P$;
    we will show that $L_{\ppyr P}(s)$ is the lifting of $L_P(s)$.

    Given a point $x$, define $\mathds 1_x(s) = [x \in sP]$
    (the ``indicator function'' of $x$);
    that is, $\mathds 1_x(s) = 1$ if $x \in sP$ and $0$ otherwise.
    Note we have $L_P = \sum_{x \in \mathbb Z^d} \mathds 1_x$.

    Observe that $\mathds 1_x$ is the indicator function of a closed interval.
    If this interval is $[a, b]$,
    denote by $\mathds 1'_x$ the result of lifting $\mathds 1_x$ at $b$.
    If the interval is $\emptyset$ or $[a, \infty)$,
    just let $\mathds 1'_x = \mathds 1_x$.
    Since we have
    \begin{equation*}
        s \ppyr P = \bigcup_{0 \leq \lambda \leq s} \lambda P,
    \end{equation*}
    we know that $x \in s \ppyr P$ whenever $s \geq a$,
    so we have $\mathds 1'_x(s) = [x \in s \ppyr P]$;
    therefore,
    \begin{equation*}
        L_{\ppyr P} = \sum_{x \in \mathbb Z^d} \mathds 1'_x.
    \end{equation*}

    It is a simple exercise showing the lifting of a sum of finitely many functions
    is the sum of their liftings.
    Let $N > 0$ be fixed.
    If we look only for $s < N$,
    only finitely many of the functions $\mathds 1_x$ will be nonzero,
    so we may apply this result.
    If $f$ is the lifting of $L_P(s)$,
    for $s < N$ we have
    \begin{equation*}
        f(s) = \sum_{x \in \mathbb Z^d} \mathds 1'_x = L_{\ppyr P}(s).
    \end{equation*}
    As $N$ was arbitrary, we conclude $L_{\ppyr P}$ is the lifting of $L_P$.

    Finally,
    if $L_P(s) = L_Q(s)$,
    then their liftings $L_{\ppyr P}(s)$ and $L_{\ppyr Q}(s)$ will be equal.
\end{proof}

In order to use this lemma,
we will decompose the pseudopyramid in several interior-disjoint pieces
and show that some of them get ``larger'' when the polytope is translated.
Since we have
\begin{equation*}
    \lim_{s \to \infty} \frac{ L_{\ppyr P}(s) }{s^d} = \vol \ppyr P,
\end{equation*}
once we show that $\ppyr P$ and $\ppyr (P + w)$ have different volumes,
Lemma~\ref{thm:different-pseudopyramid-volumes}
will guarantee that $L_P(s)$ and $L_{P + w}(s)$ are different.

If $P \subseteq \mathbb R^d$ is a full-dimensional polytope with $n$ facets,
write $P$ as
\begin{equation*}
    P = \bigcap_{i = 1}^n \{x \in \mathbb R^d \mid \langle a_i, x \rangle \leq b_i \},
\end{equation*}
so that each of its facets $F_i$ are defined by
\begin{equation*}
    F_i = P \cap \{x \in \mathbb R^d \mid \langle a_i, x \rangle = b_i \}.
\end{equation*}

\begin{figure}[t]
    \tikzset{
        every path/.style = {line join = bevel}
    }
    \centering
    \begin{subfigure}{.4\linewidth}
        \centering
        \begin{tikzpicture}
            \draw (-1, 0) -- (3, 0);
            \draw (0, -1) -- (0, 3);
            \filldraw [blue, thick, fill = blue!40]
                (1, 1) -- (2, 2) -- (2.5, 2) -- (2.5, 0.5) -- (1.5, 0.5) -- cycle;
            \draw [red, very thick]
                (2.5, 0.5) -- (1.5, 0.5) -- (1, 1);
        \end{tikzpicture}
        \caption{} % No caption
        \label{fig:back-facets}
    \end{subfigure}
    \qquad
    \begin{subfigure}{.4\linewidth}
        \centering
        \begin{tikzpicture}
            \draw (-1, 0) -- (3, 0);
            \draw (0, -1) -- (0, 3);
            \filldraw [blue, thick, fill = blue!40]
                (1, 1) -- (2, 2) -- (2.5, 2) -- (2.5, 0.5) -- (1.5, 0.5) -- cycle;
            \filldraw [blue, thick, fill = blue!40]
                (0, 0) -- (1.5, 0.5) -- (1, 1) -- cycle;
            \filldraw [blue, thick, fill = blue!40]
                (0, 0) -- (2.5, 0.5) -- (1.5, 0.5) -- cycle;
        \end{tikzpicture}
        \caption{} % No caption
        \label{fig:pseudopyramid-decomposition}
    \end{subfigure}
    \caption{
        (\subref{fig:back-facets}): Back facets (in red) of a polytope.
        (\subref{fig:pseudopyramid-decomposition}):
        Decomposition of the pseudopyramid $\ppyr P$ of a polytope $P$
        in $P$ and in pseudopyramids of its back facets.
    }
\end{figure}
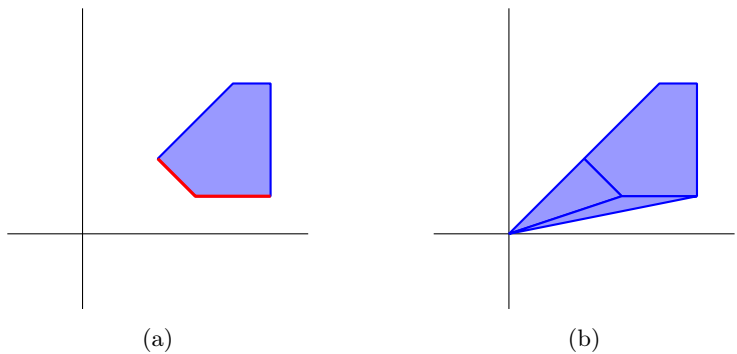

Call $F_i$ a \emph{back facet} of $P$ if $b_i < 0$
(Figure~\ref{fig:back-facets}).
The pseudopyramid $\ppyr F_i$ will intersect $P$,
but as $b_i < 0$ the interiors of these two full-dimensional polytopes are disjoint.
This idea leads to the following decomposition lemma,
whose proof is left to the reader.

\begin{lemma}
    \label{thm:pseudopyramid-decomposition}
    The pseudopyramid of a full-dimensional real polytope $P$
    is the interior-disjoint union of $P$ and the pseudopyramids $\ppyr F$
    of the back facets of $P$.
    \qed
\end{lemma}

For the next lemma,
we will also need the fact that the volume of a pyramid
is proportional to its height and to the area of its base;
more specifically,
a pyramid in $\mathbb R^d$ with height $h$
and whose base has $(d-1)$-dimensional area $A$
has volume $\frac{Ah}{d}$.

\begin{lemma}
    \label{thm:full-dimensional-translation-variant}
    Let $P$ be a full-dimensional real polytope which does not contain the origin
    and $v$ any point of $P$.
    Then for any real $\lambda > \mu \geq 0$,
    the functions $L_{P + \lambda v}(s)$ and $L_{P + \mu v}(s)$
    will differ at infinitely many points.
\end{lemma}

\begin{proof}
    Write $P$ as
    \begin{equation*}
        P = \bigcap_{i = 1}^n \{x \in \mathbb R^d \mid \langle a_i, x \rangle \leq b_i \},
    \end{equation*}
    where $n$ is the number of facets of $P$,
    so that each facet $F_i$ can be written as
    \begin{equation*}
        F_i = P \cap \{x \in \mathbb R^d \mid \langle a_i, x \rangle = b_i \}.
    \end{equation*}

    The facets of $P + \lambda v$ are of the form $F_i + \lambda v$.
    We will show that,
    for $\lambda > \mu \geq 0$,
    if $F_i + \mu v$ is a back facet of $P + \mu v$,
    then $F_i + \lambda v$ is also a back facet of $P + \lambda v$,
    and that the volume of $\ppyr(F_i + \mu v)$
    is strictly smaller than the volume of $\ppyr(F_i + \lambda v)$
    (Figure~\ref{fig:translation-of-pseudopyramid-decomposition}).
    The fact that $P$ does not contain the origin
    will guarantee the existence of at least one back facet.
    Since the volume of $P + \mu v$ and $P + \lambda v$ are the same,
    Lemma~\ref{thm:pseudopyramid-decomposition} will guarantee that
    the volume of $\ppyr(P + \mu v)$ is strictly smaller than
    the volume of $\ppyr(P + \lambda v)$,
    and thus by Lemma~\ref{thm:different-pseudopyramid-volumes}
    the functions $L_{P + \mu v}(s)$ and $L_{P + \lambda v}(s)$ are different.

    \begin{figure}[t]
        \tikzset{
            every path/.style = {line join = bevel}
        }
        \centering
        \begin{tikzpicture}
            \draw (-1, 0) -- (3, 0);
            \draw (0, -1) -- (0, 3);
            \filldraw [blue, thick, fill = blue!40]
                (.5, 1) -- (1, 2) -- (2, 2) -- (2, 0.5) -- (0.5, 0.5) -- cycle;
            \filldraw [blue, thick, fill = blue!40]
                (0, 0) -- (0.5, 0.5) -- (2, 0.5) -- cycle;
            \filldraw [blue, thick, fill = yellow!40]
                (0, 0) -- (.5, 1) -- (0.5, 0.5) -- cycle;
            \draw [red, thick] (0.5, 0.5) -- (0.5, 1) node [right] {$F_i + \mu v$};
        \end{tikzpicture}
        \qquad
        \begin{tikzpicture}
            \draw (-1, 0) -- (3, 0);
            \draw (0, -1) -- (0, 3);
            \filldraw [blue, thick, fill = blue!40]
                (1.5, 1.5) -- (2, 2.5) -- (3, 2.5) -- (3, 1) -- (1.5, 1) -- cycle;
            \filldraw [blue, thick, fill = blue!40]
                (0, 0) -- (1.5, 1) -- (3, 1) -- cycle;
            \filldraw [blue, thick, fill = blue!40]
                (0, 0) -- (1.5, 1.5) -- (2, 2.5) -- cycle;
            \filldraw [blue, thick, fill = yellow!40]
                (0, 0) -- (1.5, 1.5) -- (1.5, 1) -- cycle;
            \draw [red, thick] (1.5, 1) -- (1.5, 1.5) node [right] {$F_i + \lambda v$};
        \end{tikzpicture}
        \caption{
            For $\lambda > \mu$ and $v \in P$,
            if $F_i + \mu v$ is a back facet of $P + \mu v$,
            then $F_i + \lambda v$ is also a back facet of $P + \lambda v$.
        }
        \label{fig:translation-of-pseudopyramid-decomposition}
    \end{figure}
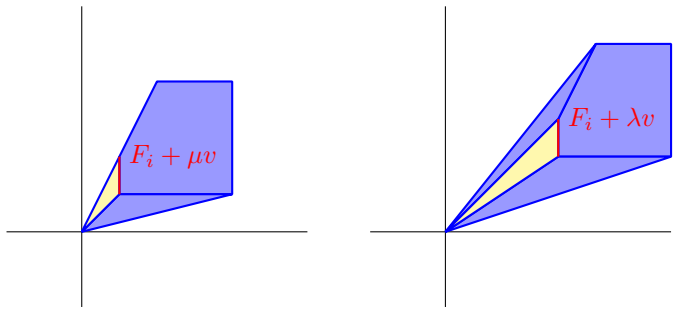

    For any $\mu$, we have
    \begin{equation*}
        P + \mu v =
            \bigcap_{i = 1}^n \{x \in \mathbb R^d \mid
                \langle a_i, x \rangle \leq b_i + \mu \langle a_i, v \rangle
            \},
    \end{equation*}
    so that
    \begin{equation*}
        F_i + \mu v = (P + \mu v) \cap
                \{x \in \mathbb R^d \mid
                    \langle a_i, x \rangle = b_i + \mu \langle a_i, v \rangle \}.
    \end{equation*}

    For all $i$,
    we know that
    \begin{equation*}
        \langle a_i, v \rangle \leq b_i,
    \end{equation*}
    because $v$ is contained in $P$,
    by assumption.
    If $F_i + \mu v$ is a back facet,
    we know that
    \begin{equation*}
        b_i + \mu \langle a_i, v \rangle < 0.
    \end{equation*}
    Adding $\mu \langle a_i, v \rangle$ to both sides of the first inequality
    and using the latter gives $\langle a_i, v \rangle < 0$.
    Therefore,
    for $\lambda > \mu \geq 0$,
    \begin{equation*}
        b_i + \lambda \langle a_i, v \rangle < b_i + \mu \langle a_i, v \rangle,
    \end{equation*}
    which shows that if $F_i + \lambda v$ is a back facet,
    then so is $F_i + \mu v$.

    As $P$ does not contain the origin,
    we know at least one of the $b_i$ is negative,
    and thus $P$ has at least one back facet $F_i$;
    therefore,
    applying the previous reasoning with $\mu = 0$ shows that
    all the polytopes $P + \lambda v$, for $\lambda \geq 0$,
    have $F_i + \lambda v$ as a back facet;
    that is,
    all these polytopes have back facets.

    Since $F_i$ is $(d-1)$-dimensional,
    the pseudopyramid $\ppyr(F_i + \mu v)$ is actually a pyramid.
    The height of this pyramid is the distance from the origin to the hyperplane
    \begin{equation*}
        \{ x \in \mathbb R^d \mid
            \langle a_i, x \rangle = b_i + \mu \langle a_i, v \rangle \}.
    \end{equation*}
    Without loss of generality we may assume $a_i$ is unitary,
    so that this distance is
    \begin{equation*}
        -\big( b_i + \mu \langle a_i, v \rangle \big).
    \end{equation*}
    As the bases of $\ppyr(F_i + \mu v)$ and $\ppyr(F_i + \mu v)$ have the same area,
    whenever $\lambda > \mu$
    the volume of $\ppyr(F_i + \mu v)$ will be strictly smaller than
    the volume of $\ppyr(F_i + \lambda v)$.

    As $P + \mu v$ has a back facet,
    by Lemma~\ref{thm:pseudopyramid-decomposition},
    the volume of $\ppyr(P + \lambda v)$
    is strictly larger than the volume of $\ppyr(P + \mu v)$.
    Observe that there might exist some facet $F_j + \lambda v$ of $P + \lambda v$
    such that $F_j + \mu v$ is not a back facet of $P + \mu v$;
    this is not a problem,
    because the back facets that do appear in $P + \mu v$
    suffice to make the volume of $\ppyr(P + \lambda v)$
    larger than $\ppyr(P + \mu v)$.

    Finally,
    since
    \begin{equation*}
        \lim_{s \to \infty} \frac{L_{\ppyr(P + \lambda v)}(s)} {s^d}
            = \vol \ppyr(P + \lambda v),
    \end{equation*}
    using Lemma~\ref{thm:different-pseudopyramid-volumes},
    we conclude that
    the functions $L_{P + \lambda v}(s)$ and $L_{P + \mu v}(s)$ must be different.
\end{proof}

\begin{restatetheorem}{thm:translation-variant}
    Let $P \subseteq \mathbb R^d$ be a real polytope
    which is either full-dimensional
    or has codimension $1$.
    Then there is an integral vector $w \in \mathbb R^d$ such that
    the functions $L_{P + k w}(s)$ are different
    for all integers $k \geq 0$.
\end{restatetheorem}

\begin{proof}
    If $P$ is full-dimensional and does not contain the origin,
    then it contains a nonzero rational vector $v$.
    Let $w$ be any nonzero multiple of $v$ which is an integer vector;
    then Proposition~\ref{thm:full-dimensional-translation-variant}
    shows directly that all the functions $L_{P + k w}(s)$,
    for $k \geq 0$,
    are different.

    If $P$ is full-dimensional, but contains the origin,
    again it will contain a nonzero rational vector $v$.
    Now $w$ not only needs to be a nonzero integer multiple of $v$,
    but also $w$ must be large enough so that $P + w$ does not contain the origin
    (such a $w$ always exist because $P$ is bounded).
    Now Proposition~\ref{thm:full-dimensional-translation-variant}
    only shows that all the functions $L_{P + k w}(s)$ will be different for $k \geq 1$.
    But since $L_{P + kw}(s)$ is nondecreasing only for $k = 0$,
    because that is the only value of $k$ for which $P + kw$ contains the origin,
    we must have $L_{P + wk}(s)$ distinct from $L_P(s)$
    whenever $k \neq 0$;
    this completes the proof in this case.

    And for the last case
    (if $P$ has codimension $1$),
    we will use Lemma~\ref{thm:different-pseudopyramid-volumes} directly.
    As $P$ is not full-dimensional,
    $P$ is contained in a hyperplane $H$ given by
    \begin{equation*}
        H = \{ x \in \mathbb R^d \mid \langle a, x \rangle = b \},
    \end{equation*}
    where $a$ is a unit vector and $b \geq 0$.

    Let $w$ be any integer vector such that $\langle a, w \rangle > 0$.
    We have
    \begin{equation*}
        P + k w \subseteq H + k w = \{
            x \in \mathbb R^d \mid
            \langle a, x \rangle = b + k \langle a, w \rangle
        \}.
    \end{equation*}

    As $P$ is not full-dimensional,
    the pseudopyramid $\ppyr(P + kw)$
    is actually a pyramid,
    whose base is $P + kw$.
    Let $A$ be the $(d-1)$-dimensional area of $P$.
    As $a$ is a unit vector,
    the height of this pyramid
    (which is the distance of $H + kw$ to the origin)
    is $b + \langle a, w \rangle$.
    Therefore,
    the volume of $\ppyr(P + kw)$ is
    \begin{equation*}
        \vol \ppyr(P + kw) = \frac1d A (b + \langle a, w \rangle).
    \end{equation*}
    Since $P$ has codimension $1$,
    its area $A$ is nonzero,
    so for $k \geq 0$ all these volumes are different.
    Thus,
    all functions $L_{P + kw}(s)$ are different in this case, too.
\end{proof}

If we assume the polytope is rational,
we may drop the dimensionality assumption.

\begin{restatetheorem}{thm:rational-translation-variant}
    Let $P \subseteq \mathbb R^d$ be a rational polytope with any dimension.
    Then there is an integral vector $w \subseteq \mathbb R^d$ such that
    the functions $L_{P + k w}(s)$ are distinct
    for all integers $k \geq 0$.
\end{restatetheorem}

\begin{proof}
    If $P$ has codimension $0$ or $1$,
    use Theorem~\ref{thm:translation-variant}.
    Otherwise,
    $P$ will be contained in a rational hyperplane passing through the origin,
    say, $H$.
    Then apply an affine transformation to $P$
    which maps $H$ to $\mathbb R^{d-1} \times \{0\}$
    and use this theorem for dimension $d-1$.
\end{proof}

We end this section by showing that,
if we do not have the rationality hypothesis,
then the dimension hypothesis is necessary.
Let $M \subseteq \mathbb R^d$ be the $(d-2)$-dimensional affine space
defined by
\begin{equation*}
    M = \{(\ln 2, \ln 3)\} \times \mathbb R^{d-2}.
\end{equation*}
That is, $M$ is the set of all points $(x_1, \dots, x_d)$ in $\mathbb R^d$
such that $x_1 = \ln 2$ and $x_2 = \ln 3$.

For any integer translation vector $w = (w_1, \dots, w_d)$
and any real $s > 0$,
we have
\begin{equation*}
    s(M + w) = \big\{\big( s(\ln 2 + w_1), s(\ln 3 + w_2) \big)\big\}
        \times \mathbb R^{d-2},
\end{equation*}
so if $s(M + w)$ contains an integer point $(x_1, \dots, x_d)$,
then $s(\ln 2 + w_1) = x_1$ and $s(\ln 3 + w_2) = x_2$.
Since $s$, $\ln 2 + w_1$ and $\ln 3 + w_2$ are nonzero,
we have $x_1, x_2 \neq 0$, too,
and thus their ratio is
\begin{equation*}
    \frac{x_1}{x_2} = \frac{\ln 2 + w_1}{\ln 3 + w_2},
\end{equation*}
which,
rearranging the terms,
gives
\begin{equation*}
    x_1 (\ln 3 + w_2) = x_2 (\ln 2 + w_1).
\end{equation*}

Raising $e$ to both sides of the equation then gives
\begin{align*}
    (e^{\ln 3 + w_2})^{x_1} &= (e^{\ln 2 + w_1})^{x_2} \\
        3^{x_1} e^{w_2 x_1} &= 2^{x_2} e^{w_1 x_2} \\
    \frac{3^{x_1}}{2^{x_2}} &= e^{w_1 x_2 - w_2 x_1}
\end{align*}

As $e$ is a transcendental number,
we must have $w_1 x_2 - w_2 x_1 = 0$,
which shows $\frac{3^{x_1}}{2^{x_2}} = 1$.
This is only possible if $x_1 = x_2 = 0$,
a contradiction.
Thus, $s(M + w)$ has no integer points.

Therefore, if $P$ is any polytope contained in $M$,
for any integer translation vector $w$ and any real $s > 0$
the polytope $s(P + w)$ will contain no integer points,
and thus $L_{P + w}(s) = 0$ for all $s > 0$ and all integer $w$.
So,
clearly these functions are all the same.

%% file: semirational-polytopes.tex
\section{Reconstruction of semi-rational polytopes}
\label{sec:semirational-polytopes}

A polytope $P$ is \emph{semi-rational}
if it can be written as
\begin{equation*}
    P = \bigcap_{i = 1}^n \{ x \in \mathbb R^d \mid \langle a_i, x \rangle \leq b_i \},
\end{equation*}
where each $a_i$ is an integer vector,
and the $b_i$ are arbitrary real numbers.
(If we demand the $b_i$ to be integers,
too,
then we recover the definition of a rational polytope.)
Every real dilation and real translation of a rational polytope
is a semi-rational polytope,
but,
for example,
$P = [0, \sqrt 2] \times [0, \sqrt 3]$
is a semi-rational polytope
which is not a translation or dilation of a rational polytope.

Suppose we know the directions $a_i$ of each half-space,
and also that we know $L_{P + w}(s)$ for all integer $w$ and all real $s > 0$.
This section will show how to extract each of the $b_i$ from this information,
effectively reconstructing the polytope.

Since we will need some technical lemmas,
we will start discussing how to reconstruct just one specific $b_i$,
but we will ``collect'' and prove each lemma where it is needed.
The complete argument is the proof of Theorem~\ref{thm:recovering-right-hand-sides}.

%% file: discontinuities-of-ehrhart-function.tex
\subsection{Discontinuities of the Ehrhart function}
\label{sec:discontinuities-of-ehrhart-function}

We will extract information about the polytope $P$
by analyzing the discontinuities of the various functions $L_{P + w}(s)$.
This section discusses the meaning of these discontinuities.

Consider the polytope $P = [\frac23, 1] \times [0, \frac 13]$
(Figure~\ref{fig:ehrhart-function-small-square}),
whose Ehrhart function is
\begin{equation*}
    L_P(s) = \left(\floor s - \ceil{\frac 23 s} + 1\right)
            \left(\floor{\frac 13 s} + 1\right).
\end{equation*}
We will analyze what happens
for the dilation parameters $s = 1$, $s = \frac 32$ and $s = 3$.

\begin{figure}[t]
    \centering
    \newcommand{\drawscaledsquare}[2]{
        % #1: scale factor
        % #2: label printed inside
        \filldraw [blue, line width = 1pt, fill = blue!40]
            (#1 * 2/3, 0) -- (#1, 0) -- (#1, #1 * 1/3) -- (#1 * 2/3, #1 * 1/3) -- cycle;

        \node at (#1 * 5/6, #1 * 1/6) {#2};
    }
    \begin{tikzpicture}[scale = 1.5]
        \draw (-1, 0) -- (4, 0);
        \draw (0, -0.5) -- (0, 2);

        \foreach \x in {-1, ..., 4}
            \foreach \y in {0, 1, 2}
                \fill[gray] (\x, \y) circle [radius = 1pt];

        \drawscaledsquare{1}{$P$};
        \drawscaledsquare{1.5}{$\frac32 P$};
        \drawscaledsquare{3}{$3P$};
    \end{tikzpicture}

    \centering
    \begin{tikzpicture}
        \begin{axis}[
            axis lines = center,
            xmin = 0,
            xmax = 3.5,
            ymin = 0,
            ymax = 4.5,
        ]
            \draw [thick]
                (0, 0) -- (1, 0)
                (1, 1) -- (1.5, 1)
                (1.5, 0) -- (2, 0)
                (2, 1) -- (3, 1)
                (3, 2) -- (3.5, 2)
                ;

            \draw [dashed]
                (1, 0) -- (1, 1)
                (1.5, 1) -- (1.5, 0)
                (2, 0) -- (2, 1)
                (3, 1) -- (3, 4)
                ;

            \addplot [closed dot] coordinates {
                (0, 1) (1, 1) (1.5, 1) (2, 1) (3, 4)
            };
            \addplot [open dot] coordinates {
                (0, 0) (1, 0) (1.5, 0) (2, 0) (3, 1) (3, 2)
            };
        \end{axis}
    \end{tikzpicture}
    \caption{
        Polytope $P = [\frac23, 1] \times [0, \frac 13]$
        and its Ehrhart function.
    }
    \label{fig:ehrhart-function-small-square}
\end{figure}
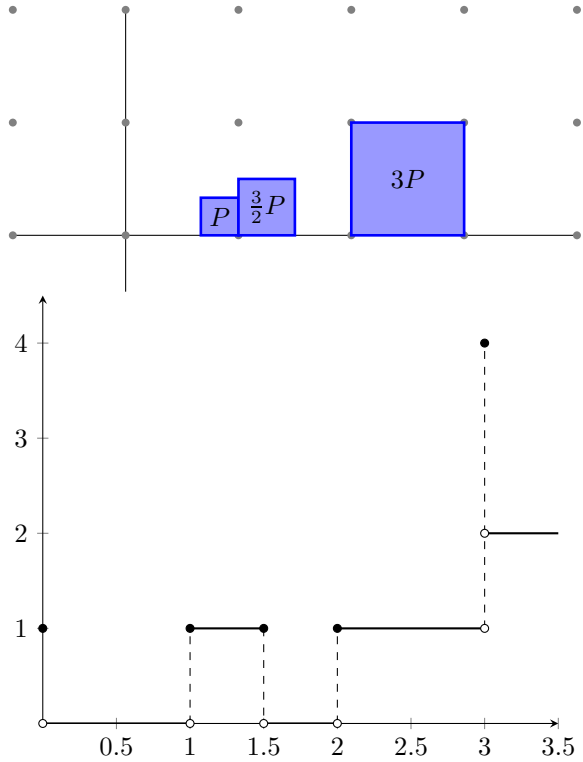

At $s = 1$,
the polytope $P$ is ``gaining'' a new integer point,
namely,
$(1, 0)$.
This gain is marked in the Ehrhart function of $P$ by a discontinuity:
$L_P(s)$ is left-discontinuous at $s = 1$.
It is a jump-discontinuity,
and the magnitude of the jump is
\begin{equation*}
    L_P(1) - \lim_{s \to 1^-} L_P(s) = 1 - 0 = 1,
\end{equation*}
which is the number of points which $P$ gains when reaching $s = 1$.

At $s = \frac32$,
the polytope ``loses'' the integer point $(1, 0)$.
Again this is marked in $L_P(s)$ by a discontinuity;
we have a right-discontinuity at $s = \frac32$,
which is again a jump-discontinuity,
and the magnitude of the jump is
\begin{equation*}
    L_P(\tfrac 32) - \lim_{s \to {\frac32}^+} L_P(s) = 1 - 0 = 1,
\end{equation*}
again the number of points lost by $P$ at $s = \frac 32$.

At $s = 3$,
these two situations happen simultaneously.
The polytope $P$ gains the points $(2, 1)$, $(3, 1)$ and $(3, 0)$,
and then immediately loses the points $(2, 1)$ and $(2, 0)$.
The gain is marked by a left-discontinuity,
with a jump of magnitude $3$,
and the loss is marked by a right-discontinuity,
with a jump of magnitude $2$.

Observe that these discontinuities are very regular:
when gaining points,
there is a left-discontinuity
and the function $L_P(s)$ ``jumps upwards'',
and when losing points,
there is a right-discontinuity
and the function $L_P(s)$ ``jumps downwards''.
The magnitude of the jump is exactly the number of points gained or lost
at that dilation parameter.
We will now formalize
how this behavior gives information about the facets of $P$.

Write $P$ as
\begin{equation*}
    P = \bigcap_{i = 1}^n \{ x \in \mathbb R^d \mid \langle a_i, x \rangle \leq b_i \},
\end{equation*}
where each $a_i$ is a primitive integer vector,
and let $F_i$ be the $i$th facet of $P$;
that is,
\begin{equation*}
    F_i = P \cap \{x \mid \langle a_i, x \rangle = b_i \}.
\end{equation*}

The facets $F_i$ for which $b_i < 0$
were called \emph{back facets} in Section~\ref{sec:translation-variant}.
By analogy,
we will call the facets $F_i$ for which $b_i > 0$ by \emph{front facets}.
The relation between the magnitude of the discontinuities
and the number of points in back and front facets
is summarized by the following lemma.

\begin{lemma}
    \label{thm:jump-magnitudes}
    Let $P$ be a full-dimensional polytope
    and $s_0$ a discontinuity point of $L_P(s)$.
    If $s_0$ is a left-discontinuity,
    then the magnitude
    \begin{equation*}
        L_P(s_0) - \lim_{s \to s_0^-} L_P(s)
    \end{equation*}
    of the jump
    is the number of integral points contained in the front facets of $s_0 P$.
    If $s_0$ is a right-discontinuity,
    then the magnitude
    \begin{equation*}
        L_P(s_0) - \lim_{s \to s_0^+} L_P(s)
    \end{equation*}
    of the jump
    is the number of integral points contained in the back facets of $s_0 P$.
\end{lemma}

\begin{proof}
    In the hyperplane representation of $P$,
    if a point $x_0$ is not contained in $s_0 P$,
    then it must violate at least one inequality;
    that is,
    \begin{equation*}
        \langle a_i, x_0 \rangle > s_0 b_i
    \end{equation*}
    for some $i$.
    Since this is a strict inequality,
    for any $s$ sufficiently close to $s_0$ we also have
    \begin{equation*}
        \langle a_i, x_0 \rangle > s b_i,
    \end{equation*}
    and thus if $x \notin s_0 P$ then $x \notin s P$
    for all $s$ sufficiently close to $s_0$.

    This means that the difference between $L_P(s_0)$
    and any of the limits
    \begin{equation*}
        L_P(s_0^+) = \lim_{s \to s_0^+} L_P(s)
        \quad \text{and} \quad
        L_P(s_0^-) = \lim_{s \to s_0^-} L_P(s)
    \end{equation*}
    must be due to points $x_0 \in s_0 P$.

    Let $x_0$ be a point in $s_0 P$.
    When considering the inequalities of $s P$ for $s < s_0$,
    if $b_i \leq 0$ we have
    \begin{equation*}
        \langle a_i, x_0 \rangle \leq s_0 b_i \leq s b_i,
    \end{equation*}
    and thus all these inequalities are satisfied.
    Thus the only inequalities that might be violated
    are the ones when $b_i > 0$,
    which correspond to front facets.

    Suppose then that $x_0$ is contained in the front facet of $s_0 P$
    which is determined by the inequality $\langle a_i, x \rangle \leq s_0 b_i$.
    The point $x_0$ satisfies this inequality with equality;
    that is,
    \begin{equation*}
        \langle a_i, x_0 \rangle = s_0 b_i.
    \end{equation*}

    If $s < s_0$, as $b_i > 0$, we have
    \begin{equation*}
        \langle a_i, x_0 \rangle = s_0 b_i > s b_i,
    \end{equation*}
    which shows that $x_0 \notin sP$ for all $s < s_0$.

    Conversely,
    if $x_0$ is not contained in any front facet of $s_0 P$,
    it will satisfy
    \begin{equation*}
        \langle a_i, x_0 \rangle < s_0 b_i
    \end{equation*}
    for all $i$ with $b_i > 0$,
    and thus for all $s < s_0$ sufficiently close to $s_0$
    the point $x_0$ will still satisfy the corresponding inequality for $s P$.
    Thus, $x_0 \in s P$ for all $s$ close enough to $s_0$.

    This means that the difference between $L_P(s_0)$ and $L_P(s_0^-)$
    must be due to integer points in front facets of $s_0 P$,
    and thus $L_P(s_0) - L_P(s_0^-)$
    is the number of integer points in front facets of $s_0 P$.

    The analysis for $s > s_0$ is analogous.
\end{proof}

For example,
the polytope $P = [\frac 23, 1] \times [0, \frac13]$
(Figure~\ref{fig:ehrhart-function-small-square})
can be written as
\begin{align*}
    P = & \phantom{{}\cap{}}
                \{(x, y) \in \mathbb R^2 \mid x \leq 1 \} \\
        & {}\cap \{(x, y) \in \mathbb R^2 \mid -x \leq \tfrac23 \} \\
        & {}\cap \{(x, y) \in \mathbb R^2 \mid y \leq \tfrac13 \} \\
        & {}\cap \{(x, y) \in \mathbb R^2 \mid -y \leq 0 \}.
\end{align*}

This polytope has two front facets,
namely $F_1 = \{1\} \times [0, \frac 13]$
(the right edge)
and $F_3 = [\frac 23, 1] \times \{\frac 13\}$
(the upper edge),
and one back facet,
namely $F_2 = \{\frac23\} \times [0, \frac 13]$
(the left edge).
The bottom edge,
$F_4 = [\frac 23, 1] \times \{0\}$,
is contained in the $x$-axis,
which is determined by the equation $y = 0$,
and thus is neither a front facet nor a back facet.

At $s = 1$,
the Ehrhart function $L_P(s)$ has a right-discontinuity,
and the magnitude of the jump there is $1$;
this corresponds to $s F_1$ containing one integer point,
namely, $(1, 0)$.

At $s = \frac 32$,
the Ehrhart function $L_P(s)$ has a left-discontinuity,
and the magnitude of the jump is also $1$;
this corresponds to $s F_2$ containing one integer point,
again $(1, 0)$.

At $s = 3$,
we have both a left and a right-discontinuity at $L_P(s)$.
The left discontinuity has magnitude $3$,
which corresponds to the three points contained in $s(F_1 \cup F_3)$
(namely, $(2, 1)$, $(3, 1)$ and $(3, 0)$),
and the right discontinuity has magnitude $2$,
which corresponds to the two points contained in $s F_2$
(namely, $(2, 1)$ and $(2, 0)$).
Note that $(3, 1)$ is not counted twice,
despite appearing in both $F_1$ and $F_3$,
whereas $(2, 1)$ is counted both as an ``entering point''
(because it is contained in the front facet $F_3$)
and as a ``leaving point''
(because it is contained in the back facet $F_2$).

%% file: relative-volumes-of-facets.tex
\subsection{Relative volumes of facets of a polytope}
\label{sec:relative-volumes-of-facets}

Let $P$ be a full-dimensional polytope.
In the proof of Lemma~\ref{thm:full-dimensional-translation-variant}
and of Theorem~\ref{thm:translation-variant},
we used the fact that
\begin{equation*}
    \lim_{s \to \infty} \frac{L_P(s)}{s^d} = \vol P.
\end{equation*}

This may be shown by noting that
\begin{align*}
    \frac{L_P(s)}{s^d}
        &= \frac{1}{s^d} \#\left( P \cap \frac 1 s \mathbb Z^d \right) \\
        &= \sum_{x \in \frac 1s \mathbb Z^d} \frac{1}{s^d} [x \in P].
\end{align*}
The last sum is,
in fact,
a Riemann sum for the indicator function of $P$,
and the fact that $P$ is Jordan-measurable
guarantees that such a sum approaches $\vol P$
by letting $s \to \infty$.

We face problems
when extending this notion to polytopes which are not full-dimensional.
For example,
consider the polytopes $F = \{1\} \times [0, 1]$
and $F' = \conv\{(1, 0), (0, 1)\}$
(Figure~\ref{fig:one-dimensional-polytopes-in-R2}).
If $s$ is an integer,
we have $L_F(s) = L_{F'}(s) = s + 1$,
and $L_F(s) = L_{F'}(s) = 0$ otherwise.
Therefore,
the analogous limits
\begin{equation*}
    \lim_{s \to \infty} \frac{L_F(s)}{s}
    \quad \text{and} \quad
    \lim_{s \to \infty} \frac{L_{F'}(s)}{s}
\end{equation*}
do not exist.

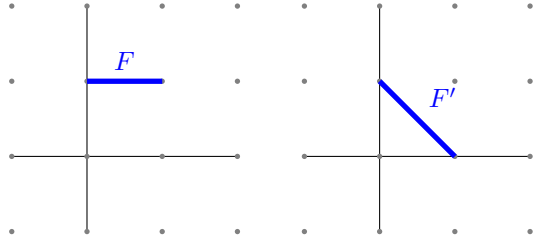
\begin{figure}[ht]
    \centering
    \begin{tikzpicture}
        \draw (-1, 0) -- (2, 0);
        \draw (0, -1) -- (0, 2);
        \foreach \x in {-1, ..., 2}
            \foreach \y in {-1, ..., 2}
                \fill [gray] (\x, \y) circle [radius = 1pt];

        \draw [line width = 2pt, blue] (0, 1) -- node [above] {$F$} (1, 1);
    \end{tikzpicture}
    \qquad
    \begin{tikzpicture}
        \draw (-1, 0) -- (2, 0);
        \draw (0, -1) -- (0, 2);
        \foreach \x in {-1, ..., 2}
            \foreach \y in {-1, ..., 2}
                \fill [gray] (\x, \y) circle [radius = 1pt];

        \draw [line width = 2pt, blue] (0, 1) -- node [above right] {$F'$} (1, 0);
    \end{tikzpicture}
    \caption{
        Two one-dimensional polytopes in $R^2$.
    }
    \label{fig:one-dimensional-polytopes-in-R2}
\end{figure}

This problem may be solved by limiting the domain
over which we take the limit.
In these examples,
we would have something like
\begin{equation*}
    \lim_{\substack{
        s \to \infty \\
        s \in \mathbb Z
    }}
    \frac{L_F(s)}{s}
    =
    \lim_{\substack{
        s \to \infty \\
        s \in \mathbb Z
    }}
    \frac{L_{F'}(s)}{s}
    =
    1.
\end{equation*}

As expected,
we don't get the lengths of $F$ or $F'$,
but their ``relative lengths''.
We will define the \emph{relative volume} of a semi-rational polytope as follows.

If $P$ is an $l$-dimensional semi-rational polytope
contained in $\mathbb R^l \times \{(0, \dots, 0)\}$,
let $P'$ be its projection to $\mathbb R^l$.
The polytope $P'$ will be a full-dimensional polytope,
and thus we define the relative volume $\rvol P$ to be $\vol P$.

If $P$ is a $l$-dimensional semi-rational polytope in $\mathbb R^d$
such that the affine span $\aff P$ contains the origin,
then $\aff P$ is a vector space,
and since $P$ is semi-rational,
$\aff P$ is, in fact,
a rational vector space
(that is, $H$ is generated by integer vectors).
Let $A$ be any unimodular transform on $\mathbb R^d$
which maps $\aff P$ to $\mathbb R^k \times \{(0, \dots, 0)\}$.
The relative volume of $P$ is then defined to be $\rvol A P$.

Finally,
if $P$ is an arbitrary semi-rational polytope,
let $v \in \aff P$ be any vector,
and define the relative volume of $P$ to be $\rvol (P - v)$.

We leave to the reader
showing that this definition does not depend on the choices of $A$ or $v$.
This is an extension of the definition found in~\cite[chapter~5.4]{ccd}
to semi-rational polytopes.

We then have the following.

\begin{lemma}
    \label{thm:limit-is-relative-volume}
    Let $P \subseteq \mathbb R^d$ be a semi-rational polytope.
    For each vector $v \in \mathbb R^d$,
    let $E_v \subseteq \mathbb R$
    be the set of all $s$ such that $\aff s(P + v)$ contains integer points.
    Then whenever $E_v$ is unbounded we have
    \begin{equation*}
        \lim_{\substack{
            s \to \infty \\
            s \in E_v
        }}
        \frac{ L_{F + v}(s) }{ s^{\dim P} } = \rvol F,
    \end{equation*}
    and the limit is uniform in $v$.
\end{lemma}

A more precise
(though less clear)
way of expressing the above limit is:
for every $\varepsilon > 0$,
there is a $N > 0$ such that,
for all vectors $v$ and all $s > N$,
if the affine span of $s(P + v)$ contains integer points,
then
\begin{equation*}
    \left| \rvol F - \frac{ L_{F + v}(s) }{ s^{\dim P} } \right| < \varepsilon.
\end{equation*}

The uniformity in $v$ will be important later.

\begin{proof}
    If $P$ is full dimensional,
    then
    \begin{align*}
        \frac{L_{P + v}(s)}{s^d}
            &= \frac{1}{s^d} \#\left( (P + v) \cap \frac 1 s \mathbb Z^d \right) \\
            &= \sum_{x \in \frac 1s \mathbb Z^d} \frac{1}{s^d} [x + v \in P],
    \end{align*}
    which is a Riemann sum for the indicator function $\mathds 1_P$ of $P$.
    Since $P$ is Jordan-measurable,
    any such sum may be made close to $\vol P$
    just by making $\frac 1 s$ small,
    regardless of the choice of $v$.
    Thus
    \begin{equation*}
        \lim_{s \to \infty} \frac{L_{P + v}(s)}{s^d} = \vol P = \rvol P,
    \end{equation*}
    and the limit is uniform in $v$.

    Next,
    assume that $P \subseteq \mathbb R^d$ is $l$-dimensional,
    that $\aff P$ contains the origin,
    and that $v \in \aff P$.
    Let $A$ be any unimodular transform
    which maps $\aff P$ to $\mathbb R^l \times \{(0, \dots, 0)\}$,
    and let $P'$ and $v'$ be the projections of $AP$ and $Av$ to $\mathbb R^l$,
    respectively.
    Then $L_{P + v}(s) = L_{P' + v'}(s)$,
    so this case reduces to the previous.

    Finally,
    let $P$ be an arbitrary semi-rational polytope
    and $v$ and arbitrary vector.
    Choose $v_0 \in \aff P$
    and define $P' = P - v_0$.
    Note that $\aff P'$ is the translation of $\aff P$
    which passes through the origin.

    Let $s \in E_v$ be fixed;
    that is,
    there exists an integer vector $w$ in $\aff s(P + v)$.
    Then $\frac 1s w - v$ is contained in $\aff P$,
    so if we let $u = \frac 1s w - v - v_0$ we have
    \begin{align*}
        L_{P + v}(s)
            &= L_{P + v - \frac1s w}(s) \\
            &= L_{P' - u}(s).
    \end{align*}

    Since $u \in \aff P'$,
    this case follows from the previous.
\end{proof}

We are most interested in the case when the polytope has codimension $1$,
because this is the case of facets of full-dimensional polytopes
(and Lemma~\ref{thm:jump-magnitudes} already hinted that this case is important).
Let $F$ be a $(d-1)$-dimensional polytope in $\mathbb R^d$
and suppose $F$ is contained in the hyperplane
\begin{equation*}
    H = \{x \in \mathbb R^d \mid \langle a, x \rangle = b \},
\end{equation*}
where $a$ is a nonnegative integer vector and $b \in \mathbb R$ is arbitrary.
We may assume $a$ is a ``primitive vector'';
that is,
there is no integer $k > 1$ such that $\frac 1k a$ is integral,
or,
equivalently,
the greatest common divisor of all entries of $a$ is $1$.
As a consequence,
the set of all possible values for $\langle a, x \rangle$,
for integer $x$,
is the set of all integers;
that is,
\begin{equation*}
    \{ \langle a, x \rangle \mid x \in \mathbb Z^d \} = \mathbb Z.
\end{equation*}

Therefore
the hyperplane $s(H + v)$ has integer points
if and only if $s(b + \langle a, v \rangle)$ is an integer.
Thus,
in this case,
Lemma~\ref{thm:limit-is-relative-volume} reads
\begin{equation*}
    \lim_{\substack{
        s \to \infty \\
        s(b + \langle a, v \rangle) \in \mathbb Z
    }}
    \frac{L_{P + v}(s)}{s^{d-1}}
    = \rvol F.
\end{equation*}

%% file: isolating-largest-vector.tex
\subsection{Isolating the facet with the largest vector}

Again,
write $P$ as
\begin{equation*}
    P = \bigcap_{i = 1}^n \{ x \in \mathbb R^d \mid \langle a_i, x \rangle \leq b_i \},
\end{equation*}
where each $a_i$ is a primitive integer vector.

A consequence of Lemma~\ref{thm:jump-magnitudes}
is that all discontinuities of $L_P(s)$
are caused by integer points passing through facets of $P$.
We will focus now on the left-discontinuities,
which are caused by integer points passing through front facets of $P$.

Let $F_i$ be the $i$th facet;
that is,
\begin{equation*}
    F_i = P \cap \{x \in \mathbb R^d \mid \langle a_i, x \rangle = b_i\}.
\end{equation*}
If $b_i > 0$,
so that $F_i$ is a front facet,
then $F_i$ is ``eligible'' for causing left-discontinuities on $P$.
Such discontinuities will only be caused by $F_i$ if $s b_i$ is an integer;
moreover,
for small values of $s$,
it might happen that $s F_i$ is too small to contain integer points.
However,
as long as $F_i$ is not a ``degenerate facet''
(that is, $F_i$ has codimension $1$),
its relative volume will be positive;
thus Lemma~\ref{thm:limit-is-relative-volume} guarantees that,
for all large enough $s$,
the polytope $s F_i$ will contain integer points
whenever $s b_i$ is an integer.

Therefore,
the discontinuities of the function $L_P(s)$ give some clues about $b_i$;
that is,
$s$ may only be a left-discontinuity point
if $s b_i$ is an integer for some $b_i > 0$,
and eventually all such $s$ are left-discontinuity points.
The problem is that these discontinuities give clues for all $b_i$ at once,
so we need a way of isolating such clues for each $b_i$.

We will look at the discontinuities of $L_{P + w}(s)$,
for a certain infinite collection of integer $w$.
If we choose $w_k = k a_1$,
for example,
then the left-discontinuities of $L_{P + w_k}(s)$
occur only when $s(b_i + k\langle a_i, a_1 \rangle)$ is an integer.
For larger $k$,
the spacing between discontinuity points of $L_{P + w_k}(s)$ decreases.
For now,
assume that $a_1$ has the largest norm among all $a_i$;
then the factor $\langle a_1, a_i \rangle$ will be largest for $i = 1$,
and thus
(for all arbitrarily large $k$)
the discontinuities coming from the facet $F_1 + w_k$
will be the closest among all facets of $P + w_k$.

However,
the fact these discontinuities are interleaved
(or even overlapping)
is what makes things difficult.
We will use the following technical lemma;
it essentially provides us with a ``window'' $(\alpha_k, \alpha_k + \epsilon_k)$
where we can,
at least infinitely often,
be sure only the discontinuities stemming from $a_1$ appear.

\begin{lemma}
    \label{thm:isolating-largest-vector}
    Let $a_1, \dots, a_n$ be primitive integer vectors in $\mathbb R^d$,
    with $\norm{a_1} \geq \norm{a_i}$ for all $i$.
    Then there is an integer vector $w_0$
    and a sequence $(\alpha_k, \alpha_k + \epsilon_k)$ of intervals
    such that,
    for all possible choice of real numbers $b_1, \dots, b_n$,
    the following properties are true:
    \begin{enumerate}[series = technical-lemma]
        \item \label{item:w_0-not-orthogonal}
            $ \langle a_i, w_0 \rangle \neq 0$ for all $i$;

        \item \label{item:w_0-same-direction-as-a_1}
            $ \langle a_1, w_0 \rangle > 0 $;

        \item \label{item:alpha_k-larger-than-k}
            $\alpha_k > k$ for all $k$;

        \item \label{item:alpha_k-tends-to-infinity}
            $\displaystyle
                \lim_{k \to \infty} |\alpha_k - k| = 0
            $;

        \item \label{item:epsilon_k-tends-to-zero}
            $\displaystyle
                \lim_{k \to \infty} \epsilon_k = 0
            $;

        \item \label{item:a_1-provokes-discontinuities}
            For all sufficiently large $k$,
            there is either one or two distinct values of $s$
            in $(\alpha_k, \alpha_k + \epsilon_k)$
            such that
            \begin{equation*}
                s(b_1 + k \langle a_1, w_0 \rangle)
            \end{equation*}
            is an integer;
            and

        \item \label{item:no-other-discontinuities}
            There exists infinitely many $k$ such that,
            for all $i \geq 2$ such that $\langle a_i, w_0 \rangle > 0$,
            there is no $s \in (\alpha_k, \alpha_k + \epsilon_k)$ such that
            \begin{equation*}
                s(b_i + k \langle a_i, w_0 \rangle)
            \end{equation*}
            is an integer.
    \end{enumerate}
\end{lemma}

Before proving the lemma,
let us see what these properties mean,
intuitively.

Properties \ref{item:alpha_k-larger-than-k},
\ref{item:alpha_k-tends-to-infinity} and~\ref{item:epsilon_k-tends-to-zero}
says that the numbers $s \in (\alpha_k, \alpha_k + \epsilon_k)$
are of the form $k + \delta$,
where $\delta$ is a positive value
which approaches zero for large $k$.
This will guarantee the hypothesis of Lemma~\ref{thm:pseudodiophantine-equations}.

Define $w_k = k w_0$,
so that the interval $(\alpha_k, \alpha_k + \epsilon_k)$
is an ``interesting interval'' of discontinuities in $L_{P + w_k}(s)$.
We are assuming we know the vectors $a_i$,
but not the right-hand sides $b_i$.
Property~\ref{item:w_0-not-orthogonal}
guarantees that,
for all sufficiently large $k$,
the right-hand sides $b_i + \langle a_i, w_k \rangle$
will have the sign of $\langle a_i, w_0 \rangle$,
so we know that all left-discontinuities of $L_{P + w_k}(s)$
are caused by the $a_i$ for which $\langle a_i, w_0 \rangle > 0$.
Property~\ref{item:w_0-same-direction-as-a_1} says that
$a_1$ is one of these vectors which causes the left-discontinuity.

Define $V_k$ to be the sum of the magnitudes
of all left-discontinuities of $L_{P + w_k}(s)$
which happen in the interval $(\alpha_k, \alpha_k + \epsilon_k)$.
Property~\ref{item:a_1-provokes-discontinuities}
says that,
for all large enough $k$,
at least one of these discontinuities is caused by $F_1$;
Lemma~\ref{thm:limit-is-relative-volume}
and the fact that the values in $(\alpha_k, \alpha_k + \epsilon_k)$ are approximately $k$
says that this discontinuity has magnitude of approximately $k^{d-1} \rvol F_1$.
(Here we use the uniformity in $w$ claimed by that lemma.)

Intuitively,
each $V_k$ equals $k^{d-1} \rvol F_1$
plus some positive garbage.
Properties \ref{item:a_1-provokes-discontinuities}
and~\ref{item:no-other-discontinuities}
allows us to handle this garbage.

We may categorize the $V_k$ in three cases:
the \emph{good} case,
where $L_{P + w_k}(s)$ has just a single discontinuity
in $(\alpha_k, \alpha_k + \epsilon_k)$;
the \emph{not-so-good} case,
where there is two discontinuities of about the same magnitude,
and the \emph{bad} case,
which is the remaining cases.
Property~\ref{item:a_1-provokes-discontinuities}
says that,
in the good and the not-so-good case,
$V_k$ is approximately $k^{d-1} \rvol F_1$ and $2 k^{d-1} \rvol F_1$,
respectively,
and Property~\ref{item:no-other-discontinuities}
says that either the good or the not-so-good case happen infinitely often.

Therefore,
at least one of the following limits is true
(that is, exists and equals the expression in the right-hand side):
\begin{equation*}
    \lim_{\substack{
        k \to \infty \\
        \text{$V_k$ good}
    }}
    \frac{V_k}{k^{d-1}}
    = \rvol F_1
    \qquad
    \text{or}
    \qquad
    \lim_{\substack{
        k \to \infty \\
        \text{$V_k$ not-so-good}
    }}
    \frac{V_k}{k^{d-1}}
    = 2 \rvol F_1.
\end{equation*}

Therefore,
if we know $a_1, \dots, a_n$ and $L_{P + w}(s)$ for all integer $w$,
then we can determine $\rvol F_1$.
Later,
we will see how to do some sort of induction to get the values of $\rvol F_i$
for the remaining $i$, too;
this is why Lemma~\ref{thm:isolating-largest-vector}
was stated directly in terms of the vectors $a_i$,
without referring to the polytope.

\begin{proof}[Proof of Lemma~\ref{thm:isolating-largest-vector}]
    First,
    choose an integral vector $w' \in \{a_1\}^\perp$
    which is not orthogonal to any of the vectors $a_i$,
    for $a_i \neq \pm a_1$
    (such a $w'$ exist because the intersections $\{a_1\}^\perp \cap \{a_i\}^\perp$,
    where the ``forbidden'' $w'$ falls,
    have codimension $2$,
    and thus their union are a proper subset of $\{a_1\}^\perp$).

    Since
    \begin{equation*}
        \langle a_1, a_i \rangle < \norm{a_1} \norm{a_i} \leq \norm{a_1}^2,
    \end{equation*}
    for all large enough $\tau > 0$ we have
    \begin{equation*}
        \tau \norm{a_1}^2 > \langle a_i, w' + \tau a_1 \rangle.
    \end{equation*}

    So, choose $\tau > 0$ to be an integer so large that
    the above equation is satisfied,
    and also that
    \begin{equation*}
        \langle a_i, w' + \tau a_1 \rangle \neq 0
    \end{equation*}
    for all $i$.
    (If $\langle a_i, a_1 \rangle = 0$,
    then any $\tau$ will do,
    due to the choice of $w'$;
    otherwise,
    we may just make $\tau$ large,
    because the other terms in the expression are constant.)

    Define $w_0 = \tau a_1 + w'$.
    This definition of $w_0$ satisfies conditions
    \ref{item:w_0-not-orthogonal} and~\ref{item:w_0-same-direction-as-a_1}.

    Define $\epsilon_k = \frac1k \epsilon_0$,
    where $\epsilon_0$ is a value
    (to be determined later)
    which satisfies
    \begin{equation}
        \frac{1}{\langle a_1, w_0 \rangle}
        < \epsilon_0 <
        \frac{2}{\langle a_1, w_0 \rangle}.
        \label{eq:bounds-on-epsilon_0}
    \end{equation}

    This suffice to get property~\ref{item:a_1-provokes-discontinuities}:
    note that $s(b_1 + k\langle a_1, w_0 \rangle)$ is an integer
    if and only if $s = \frac{m}{b_1 + k\langle a_1, w_0 \rangle}$
    for some integer $m$.
    Therefore,
    any open interval in $\mathbb R$
    whose length is larger than $\frac{1}{b_1 + k \langle a_1, w_0 \rangle}$
    is guaranteed to contain at least one of these.
    The first inequality guarantees that
    \begin{equation*}
        \frac{1}{b_1 + k \langle a_1, w_0 \rangle} < \frac{\epsilon_0}{k} = \epsilon_k
    \end{equation*}
    for all large enough $k$.
    Since $(\alpha_k, \alpha_k + \epsilon_k)$ has lenght $\epsilon_k$,
    this guarantees that,
    for all large enough $k$,
    at least one $s$ in this interval satisfies
    $s(b_1 + k \langle a_1, w_0 \rangle) \in \mathbb Z$.

    Now,
    any open interval which contains three distinct numbers $s$ of the form
    $\frac{n}{b_1 + k\langle a_1, w_0 \rangle}$,
    for integer $n$,
    must have length larger than $\frac{2}{b_1 + k \langle a_1, w_0 \rangle}$.
    Here,
    the second inequality in~\eqref{eq:bounds-on-epsilon_0}
    guarantees that
    \begin{equation*}
        \epsilon_k = \frac{\epsilon_0}{k} < \frac{2}{b_1 + k \langle a_1, w_0 \rangle}
    \end{equation*}
    for all large enough $k$.
    This shows that,
    once we define $\epsilon_0$ properly
    (satisfying inequality~\eqref{eq:bounds-on-epsilon_0}),
    we satisfy both properties \ref{item:a_1-provokes-discontinuities}
    and~\ref{item:epsilon_k-tends-to-zero}.

    Property~\ref{item:no-other-discontinuities} will be the hardest.
    We will define $\alpha_k$ to be a value which is a bit greater that $k$,
    so that the values of the interval $(\alpha_k, \alpha_k + \epsilon_k)$
    are of the form $k + \delta$,
    where $\delta$ is ``small, but not too small''.
    The value $\delta$ will lie in $(\alpha_k - k, \alpha_k - k + \epsilon_k)$.
    We will make $\alpha_k - k + \epsilon_k$ close to $0$,
    so that $\delta$ will be small,
    but we will make sure $\alpha_k - k$ stays some ``safe distance'' from $0$,
    so that $\delta$ is not ``too small''.

    We want to control when $s(b_i + k \langle a_i, w_0 \rangle)$ is an integer.
    If $s = k + \delta$,
    this number is
    \begin{equation*}
        (k + \delta)(b_i + k \langle a_i, w_0 \rangle) =
            k^2 \langle a_i, w_0 \rangle +
            \delta b_i +
            k b_i +
            \delta k \langle a_i, w_0 \rangle.
    \end{equation*}
    If we can guarantee that this value is distant from an integer
    for all $i > 2$,
    we get property~\ref{item:no-other-discontinuities}.

    The first term of the above expression,
    $k^2 \langle a_i, w_0 \rangle$,
    will always bee an integer,
    so we have nothing to do.

    We will choose $\alpha_k$ and $\epsilon_k$
    so that $\alpha_k - k$ is proportional to $\frac 1k$.
    As $\epsilon_k = \frac{\epsilon_0}{k}$ is also proportional to $\frac 1k$,
    and $\delta$ is ``sandwiched'' between these two values,
    the value of $\delta$ will be proportional to $\frac 1 k$, too.
    This means that the second term,
    $\delta b_i$,
    will tend to zero for large $k$.

    The third term,
    $k b_i$,
    will be dealt with in an indirect way.
    By the ``simultaneous Dirichlet's approximation theorem'',
    for each $N$ there is a $k \in \{1, \dots, N^n\}$ such that
    all numbers $k b_i$ are within $\frac 1 N$ of an integer.
    By choosing larger and larger $N$,
    we guarantee the existence of arbitrarily large $k$
    where all the distance of $k b_i$ to the nearest integer
    is made arbitrarily small.
    (This indirect attack to $k b_i$ is the responsible for the phrase
    ``there exist infinitely many $k$''
    in the condition~\ref{item:no-other-discontinuities},
    as opposed to something like
    ``for all large enough $k$''.)

    So,
    the difficulties rests upon the fourth term,
    $\delta k \langle a_i, w_0 \rangle$.
    We want to place it in the interval $(\varepsilon, 1 - \varepsilon)$
    for some $\varepsilon > 0$,
    so that this term is always at least within $\varepsilon$ of the nearest integer.
    Since for the other three terms the distance to the nearest integer
    gets arbitrarily small,
    but this fourth term stays distant,
    we can guarantee that
    $(k + \delta)(\langle a_i, w_0 \rangle k + b_i)$ is not an integer.

    We will determine a suitable $\varepsilon > 0$ later.
    Let $\varepsilon' > 0$ be such that,
    for $i \geq 2$,
    we have $\langle a_i, w_0 \rangle > \varepsilon'$
    whenever $\langle a_i, w_0 \rangle > 0$,
    and $\langle a_i, w_0 \rangle < \tau \norm{a_1}^2 - \varepsilon'$
    for all $i \geq 2$.

    Define
    \begin{equation*}
        \alpha_k = k + \frac{1}{k} \frac{\varepsilon}{\varepsilon'}.
    \end{equation*}

    Observe that this definition of $\alpha_k$
    clearly satisfies conditions \ref{item:alpha_k-larger-than-k}
    and~\ref{item:alpha_k-tends-to-infinity}.
    We have $\delta > \frac{1}{k} \frac{\varepsilon}{\varepsilon'}$,
    and thus
    \begin{equation*}
        \delta k \langle a_i, w_0 \rangle > \delta k \varepsilon' > \varepsilon.
    \end{equation*}

    Define
    \begin{equation*}
        \epsilon_0 = \frac{1 - \varepsilon}{t \norm{a_1}^2 - \varepsilon'}
            - \frac{\varepsilon}{\varepsilon'}.
    \end{equation*}
    Then since $\epsilon_k = \frac{\epsilon_0}{k}$,
    we get
    \begin{equation*}
        \delta < \alpha_k - k + \epsilon_k
        = \frac{1}{k} \frac{1 - \varepsilon}{t \norm{a_1}^2 - \varepsilon'},
    \end{equation*}
    and thus
    \begin{equation*}
        \delta k \langle a_i, w_0 \rangle < \delta k (t \norm{a_1}^2 - \varepsilon')
        < 1 - \varepsilon.
    \end{equation*}

    This show that,
    regardless of the value of $\varepsilon > 0$,
    the definitions of $\alpha_k$ and $\epsilon_k$ guarantee that,
    if $s = k + \delta \in (\alpha_k, \alpha_k + \epsilon_k)$,
    then $\delta k \langle a_i, w_0 \rangle \in (\varepsilon, 1 - \varepsilon)$.
    The other terms in $s (b_i + k \langle a_i, w_0 \rangle)$
    will approximate integer values,
    thus showing property~\ref{item:no-other-discontinuities}.

    Now we will choose $\varepsilon > 0$
    so that the definition of $\epsilon_0$
    satisfy inequality~\eqref{eq:bounds-on-epsilon_0}.

    The inequality $\frac{1}{\langle a_1, w_0 \rangle} < \epsilon_0$
    can be rewritten as follows:
    \begin{align*}
        \epsilon_0 &> \frac{1}{\langle a_1, w_0 \rangle} \\
        \frac{1 - \varepsilon}{\tau \norm{a_1}^2 - \varepsilon'}
            - \frac{\varepsilon}{\varepsilon'}
            &> \frac{1}{\tau \norm{a_1}^2}
            \\
        \frac{1}{\tau \norm{a_1}^2 - \varepsilon'} &>
            \frac{1}{\tau \norm{a_1}^2} +
            \varepsilon \left(
                \frac{1}{\varepsilon'} +
                \frac{1}{\tau \norm{a_1}^2 - \varepsilon'}
            \right).
        \intertext{
            Analogously,
            the inequality $\epsilon_0 < \frac{2}{\langle a_1, w_0 \rangle}$
            can be rewritten as
        }
        \frac{1}{\tau \norm{a_1}^2 - \varepsilon'} &<
            \frac{2}{\tau \norm{a_1}^2} +
            \varepsilon \left(
                \frac{1}{\varepsilon'} +
                \frac{1}{\tau \norm{a_1}^2 - \varepsilon'}
            \right).
    \end{align*}

    As $\varepsilon' > 0$ is a small value,
    it is clear that
    \begin{equation*}
        \frac{1}{\tau \norm{a_1}^2}
        <
        \frac{1}{\tau \norm{a_1}^2 - \varepsilon'}
        <
        \frac{2}{\tau \norm{a_1}^2},
    \end{equation*}
    so an $\varepsilon > 0$ satisfying both inequalities above does exist.
    This guarantees condition~\ref{item:a_1-provokes-discontinuities},
    finishing the proof of the lemma.
\end{proof}

%% file: pseudodiophantine-equations.tex
\subsection{Pseudo-Diophantine equations}

We will now show how to recover $b_1$.
From the discussion preceding the proof of Lemma~\ref{thm:isolating-largest-vector},
we have a sequence of values $V_k$
from which we extract $\rvol F_1$.
The $V_k$ represent the values of some discontinuities of $L_{P + w_k}(s)$
which we know are due to points passing through $F_1$.
We are interested
in these discontinuities.

For simplicity,
we will assume that there are infinitely many good $V_k$.
For all large enough $k$ for which $V_k$ is good,
there is precisely one real number $s_k$ in the interval
$(\alpha_k, \alpha_k + \epsilon_k)$
which corresponds to a discontinuity of $L_{P + w_k}(s)$,
and this discontinuity was caused by $F_1$.
That is,
we have infinitely many equations of the following form:
there is some integer $m_k$ such that
\begin{equation*}
    s_k (b_1 + k \langle a_1, w_0 \rangle) = m_k.
\end{equation*}

The following lemma asserts that
there is exactly one solution for $b_1$
of this infinite set of ``semi-Diophantine equations''.

\begin{lemma}
    \label{thm:pseudodiophantine-equations}
    Let $s_l \in \mathbb R$, $k_l \in \mathbb Z$ be given for each integer $l > 0$,
    where the $s_l$ are non-integer numbers
    such that $\{s_l\}$ approaches zero for $l \to \infty$.
    Then the infinite system of equations
    \begin{equation*}
        s_l (b + k_l) = m_l,
        \qquad
        \text{$l > 0$}
    \end{equation*}
    for $b \in \mathbb R$ and $m_l \in \mathbb Z$,
    has at most one solution.
\end{lemma}

\begin{proof}
    Without loss of generality, we may assume all $s_l$ and all $k_l$ are different.
    Let $(b, m_1, m_2, m_3, \dots)$ be a solution for this system of equations.
    If we did not have the integrality constraint on $m_l$,
    then all the solutions of this infinite system of equations
    would have the form
    \begin{equation*}
        (b + \lambda, m_1 + s_1 \lambda, m_2 + s_2 \lambda, \dots),
        \qquad \lambda \in \mathbb R.
    \end{equation*}

    The integrality constraint dictates that $m_l + s_l \lambda$ is an integer,
    say $m'_l$;
    thus for some integer $\nu_l = m'_l - m_l$ we have $s_l \lambda = \nu_l$.

    Assume $b$ is irrational;
    then all $s_l$ are also irrational.
    Since
    \begin{equation*}
        \lambda = \frac{\nu_1}{s_1} = \frac{\nu_2}{s_2},
    \end{equation*}
    so we know that $\frac{s_1}{s_2} = \frac{\nu_1}{\nu_2}$ is a rational number.
    Since $b = \frac{m_1}{s_1} - k_1 = \frac{m_2}{s_2} - k_2$,
    we have
    \begin{align*}
        k_2 - k_1 &= \frac{m_1}{s_1} - \frac{m_2}{s_2} \\
        s_1( k_2 - k_1 ) &= m_1 - \frac{s_1}{s_2} m_2.
    \end{align*}

    Since $k_1 \neq k_2$,
    the above equation shows that $s_1$ is rational;
    this is a contradiction,
    unless $\nu_2 = 0$,
    which amounts to $\lambda = 0$.
    Therefore,
    if $b$ is irrational,
    we indeed have at most one solution.

    Assume now that $b$ is rational.
    Therefore,
    all $s_l$ are rational,
    so we may write $s_l = \frac{p_l}{q_l}$
    for relatively prime $p_l, q_l$.

    For $m_l + s_l \lambda$ to be an integer,
    $\lambda$ must be a rational number,
    say, $\lambda = \frac p q$.
    The number $s_l \lambda$ must also be an integer,
    say $r$;
    that is,
    \begin{equation*}
        p_l p = q_l q r.
    \end{equation*}
    Reducing both sides modulo $q_l$ gives
    \begin{equation*}
        p_l p \equiv 0 \mod{q_l};
    \end{equation*}
    as $p_l$ and $q_l$ are relatively prime,
    this shows that $q_l$ divides $p$.

    Finally,
    since we assumed that no $s_l$ is an integer,
    but the distance of $s_l$ to the nearest integer approaches zero as $l \to \infty$,
    we must have $\lim_{l \to \infty} q_l = \infty$.
    So $p$ is a multiple of arbitrarily large numbers,
    thus the only possibility is $p = 0$,
    which implies $\lambda = 0$,
    from which the lemma follows.
\end{proof}

The fact that each $s_k$ is contained in $(\alpha_k, \alpha_k + \epsilon_k)$,
together with properties~\ref{item:alpha_k-larger-than-k},
\ref{item:alpha_k-tends-to-infinity} and~\ref{item:epsilon_k-tends-to-zero}
from Lemma~\ref{thm:isolating-largest-vector},
guarantees that the fractional part $\{s_k\}$ of $s_k$,
although never zero,
approaches zero as $k \to \infty$.
Therefore,
we are in position to apply the lemma,
which thus determines $b_1$ uniquely.

%% file: piecing-together-rational-case.tex
\subsection{Piecing together the semi-rational case}

Throughout the last sections,
we argued how we could obtain $b_1$ and $\rvol F_1$
if we knew $a_1, \dots, a_n$ and $L_{P + w}(s)$
for all integer $w$ and all real $s > 0$.
During this discussion,
we collected some lemmas,
which we now will use to show how to obtain the remaining $b_i$.

\begin{theorem}
    \label{thm:recovering-right-hand-sides}
    Let $a_1, \dots, a_n \in \mathbb R^d$ be primitive integer vectors
    and for each integer $w$ let $f_w(s)$ be a function on $\mathbb R$.
    Then there is at most one set of numbers $b_1, \dots, b_n$ such that
    \begin{equation*}
        P = \bigcap_{i=1}^n \{x \in \mathbb R^d \mid \langle a_i, x \rangle \leq b_i\}
    \end{equation*}
    is a full-dimensional semi-rational polytope,
    that $f_w(s) = L_{P + w}(s)$ for all $s > 0$ and all integer $w$,
    and that the polytopes $F_i$ defined by
    \begin{equation*}
        F_i = P \cap \{x \in \mathbb R^d \mid \langle a_i, x \rangle = b_i\}
    \end{equation*}
    are faces of $P$.
\end{theorem}

Observe that here we are not assuming that $n$ is the number of facets of $P$;
in particular,
every facet of $P$ will be an $F_i$ for some $i$,
but some of the $F_i$ might be lower-dimensional faces of $P$.
This will be important
in the proof of Corollary~\ref{thm:semirational-complete-invariant}.

\begin{proof}
    Order the vectors $a_i$ in nonincreasing order of length;
    that is,
    \begin{equation*}
        \norm{a_1} \geq \norm{a_2} \geq \dots \geq \norm{a_n}.
    \end{equation*}

    For each $i$,
    we will determine whether $F_i$ is a facet of $P$ or not
    (that is, whether $F_i$ has codimension $1$),
    and,
    in the case it is a facet,
    we will determine $\rvol F_i$ and $b_i$.
    Then,
    given $b_i$ for the facets of $P$,
    the polytope is completely determined;
    since then $b_i$ is the smallest real number such that
    \begin{equation*}
        P \subseteq \{ x \in \mathbb R^d \mid \langle a_i, x \rangle \leq b_i \},
    \end{equation*}
    this determines the remaining $b_i$.

    We will proceed by induction,
    so assume that,
    for some $j \geq 1$,
    we have determined everything for the faces $F_i$ with $i < j$.

    By Lemma~\ref{thm:jump-magnitudes},
    all discontinuities of the function $L_{P + w}(s)$
    are due to integer points passing through the facets of $P$.
    If $s_0$ is a point of left-discontinuity of $L_{P + w}(s)$,
    let $F_{i_1} + w, \dots, F_{i_l} + w$ be the front facets of $P + w$
    such that $s_0 (b_i + \langle a_i, w \rangle)$ is an integer;
    that is, $F_{i_1}, \dots, F_{i_l}$
    are the facets which caused the discontinuity at $s_0$.
    Then Lemma~\ref{thm:limit-is-relative-volume} says that
    if $U$ is the magnitude of this discontinuity,
    then $\frac{U}{s_0^{d-1}}$
    is approximately $\rvol F_{i_1} + \dots + \rvol F_{i_l}$.

    Here,
    the fact that Lemma~\ref{thm:limit-is-relative-volume}
    asserts that the limit is uniform in the translation vector $w$
    guarantees that
    the number $\frac{U}{s^{d-1}}$
    will approximate $\rvol F_{i_1} + \dots + \rvol F_{i_l}$
    just by making $s$ large,
    regardless of $w$.

    Let $S(x) = 0^{d-1} + 1^{d-1} + \dots + \floor{x}^{d-1}$ for $x \geq 0$
    and $S(x) = -S(-x)$ for $x \leq 0$,
    and define
    \begin{equation*}
        g_w(s) = f_w(s) - \sum_{i=1}^{j-1} S(s(b_i + \langle a_i, w \rangle)) \gamma_i,
    \end{equation*}
    where $\gamma_i$ is $\rvol F_i$,
    if $F_i$ is a facet of $P$,
    and zero otherwise.

    If $x = s(b_i + \langle a_i, w \rangle)$ is a positive integer,
    then the term $S(x) \rvol F_i$
    will cause a left-discontinuity of magnitude $x^{d-1} \rvol F_i$;
    if $x$ is a negative integer,
    then the term $S(x) \rvol F_i$
    will cause a right-discontinuity of magnitude $x^{d-1} \rvol F_i$.
    Subtracting these values from $f_w$ ``cleans'' the function
    from the influence of the discontinuities caused by $F_i$,
    for $i < j$.

    In terms of $\frac{U}{s^{d-1}}$,
    this says that,
    if $V$ is the magnitude of the left discontinuity at $s_0$ of $g_w$,
    then $\frac{V}{s_0^{d-1}}$ approximates the sum of the $\rvol F_{i_m}$
    for which $i_m \geq j$.

    Since we know whether $F_i$ is a facet or not,
    and we know $\rvol F_i$ and $b_i$ in the case it is,
    we may construct the function $g_w$,
    so we may work with the ``cleaner'' $V$ instead of with $U$.
    (Note that removing the terms $S(x) \rvol F_i$
    does not perfectly eliminate the influence of the facets $F_i$,
    because the magnitudes of the jumps only approximate $s^{d-1} \rvol F_i$.
    Therefore,
    there might be some ``residual'' discontinuities of order $O(s^{d-2})$ in $g_w$.
    Since we will always divide the magnitude of the discontinuities by $s^{d-1}$,
    we may safely ignore these residual discontinuities.)

    Use Lemma~\ref{thm:isolating-largest-vector}
    with the vectors $a_j, \dots, a_n$
    to get appropriate $w_0$, $\alpha_k$ and $\epsilon_k$.

    Define $V_k$ to be sum of the magnitudes of the left-discontinuities
    of the function $g_k$ in the interval $(\alpha_k, \alpha_k + \epsilon_k)$.
    We will first determine whether $F_j$ is a facet or not.

    Property~\ref{item:a_1-provokes-discontinuities} says that,
    if $F_j$ is a facet
    (so that its $(d-1)$-dimensional volume is nonzero),
    the value of $\frac{V_k}{k^{d-1}}$
    must be at least $\rvol F_j$.
    Property~\ref{item:no-other-discontinuities} says that,
    for arbitrarily many $k$,
    the value of $\frac{V_k}{k^{d-1}}$
    will be either $\rvol F_j$ or $2 \rvol F_j$.
    Conversely,
    if $F_j$ is not a facet,
    then this number will approach zero.
    Therefore,
    $F_j$ is a facet if and only if
    \begin{equation*}
        \liminf_{k \to \infty} \frac{V_k}{k^{d-1}} > 0.
    \end{equation*}

    If $F_j$ is not a facet,
    there is nothing more to do,
    so assume that $F_j$ is a facet of $P$.

    Let us say that a value $V_k$ is \emph{good}
    if the interval $(\alpha_k, \alpha_k + \epsilon_k)$
    contains exactly one discontinuity;
    \emph{not-so-good},
    if the interval contains exactly two discontinuities,
    and the magnitude of the jump in both cases is approximately the same
    (that is, their ratio is close to $1$);
    and \emph{bad} otherwise.

    If $V_k$ is good,
    then we know from Lemma~\ref{thm:limit-is-relative-volume}
    that $V_k$ is approximately $k^{d-1} \rvol F_j$.
    If $V_k$ is not-so-good,
    we know at least one of the two discontinuities
    must have a magnitude of approximately $k^{d-1} \rvol F_j$,
    so the other discontinuity must also have that magnitude,
    and thus $V_k$ is approximately $2 k^{d-1} \rvol F_j$.

    Property~\ref{item:no-other-discontinuities}
    of Lemma~\ref{thm:isolating-largest-vector} says that
    there are infinitely many $V_k$ which are either good or not-so-good.
    This means that at least one of the limits
    \begin{equation*}
        \liminf_{\substack{
                k \to \infty \\
                \text{$V_k$ good}
            }}
            \frac{V_k}{k^{d-1}}
        \qquad
        \text{and}
        \qquad
        \frac12
        \liminf_{\substack{
                k \to \infty \\
                \text{$V_k$ not-so-good}
            }}
            \frac{V_k}{k^{d-1}}
    \end{equation*}
    exists and equals $\rvol F_j$.

    If either one of the limits do not exist
    (which happens only if there is only finitely many good $V_k$
    or finitely many not-so-good $V_k$, respectively),
    or both limits agree,
    then we know for sure the value of $\rvol F_j$.
    It might happen that both limit exists,
    but their values are different;
    this happens if there are infinitely many good and not-so-good $V_k$,
    but either of these have some associated ``garbage''.
    For good $V_k$,
    we know that there is a single discontinuity in $(\alpha_k, \alpha_k + \epsilon_k)$,
    but it might happen that this discontinuity
    is due to $F_j$ and some other facets $F_i$ for $i > j$.
    A similar problem happens with the not-so-good $V_k$.
    In this case,
    we will have ``dirty'' $V_k$,
    which will make the limits larger than $\rvol F_j$.

    However,
    property~\ref{item:no-other-discontinuities}
    does guarantee that
    there will be infinitely many ``clean'' $V_k$,
    so in the event that both limits exist and differ,
    the smallest value is the correct one
    (because that's where the infinitely many ``clean'' $V_k$ appeared).

    To recover $b_j$,
    we will handle these cases separately.

    Assume first that the ``infinitely many good $V_k$'' gave the correct value.
    Looking in the intervals $(\alpha_k, \alpha_k + \epsilon_k)$
    for which the corresponding $V_k$ approximate $k^{d-1} \rvol F_j$,
    we get an infinite number of $s_k \in (\alpha_k, \alpha_k + \epsilon_k)$
    which we know are discontinuities provoked by $F_j$.
    That is,
    there is a sequence of integers $m_k$ such that
    \begin{equation*}
        s_k (b_j + k \langle a_j, w_0 \rangle) = m_k
    \end{equation*}
    for all these $k$.
    By properties \ref{item:alpha_k-larger-than-k},
    \ref{item:alpha_k-tends-to-infinity} and~\ref{item:epsilon_k-tends-to-zero}
    of Lemma~\ref{thm:isolating-largest-vector},
    we have that $\{s_k\}$ approaches zero for these $k$,
    and thus we may apply Lemma~\ref{thm:pseudodiophantine-equations}
    to determine $b_j$ uniquely.

    Now assume that the correct limit is the ``infinitely many not-so-good $V_k$''.
    Let $s_k$ and $s'_k$ (with $s_k < s'_k$)
    be the two discontinuities which happen in $(\alpha_k, \alpha_k, \epsilon_k)$,
    when $V_k$ is not-so-good and approximates $2 \rvol F_j$.
    The main difficulty of this case is that,
    for each $k$,
    we know at least one of $s_k$ and $s'_k$
    is a point of discontinuity caused by $F_j$,
    but it might not be both.

    Call the case where both are discontinuities caused by $F_j$
    the \emph{nice} case.
    Since these two are consecutive discontinuities,
    we have
    \begin{equation*}
        s'_k - s_k = \frac{1}{b_j + k \langle a_j, w_0 \rangle};
    \end{equation*}
    rearranging this equation gives
    \begin{equation*}
        b_j = \frac{1}{s'_k - s_k} - k \langle a_j, w_0 \rangle.
    \end{equation*}

    In a non-nice case,
    the distance between $s'_k$ and $s_k$
    is smaller than in the nice case
    (because in any open interval
    with length larger than $\frac{1}{b_j + k \langle a_j, w_0 \rangle}$
    lies a point of discontinuity of $F_j$);
    that is,
    \begin{equation*}
        s'_k - s_k < \frac{1}{b_j + k \langle a_j, w_0 \rangle}.
    \end{equation*}
    If $k$ is large enough,
    this translates to
    \begin{equation*}
        b_j < \frac{1}{s'_k - s_k} - k \langle a_j, w_0 \rangle.
    \end{equation*}

    Due to properties \ref{item:a_1-provokes-discontinuities}
    and~\ref{item:no-other-discontinuities},
    together with the current assumption that
    there are only finitely many good $V_k$,
    we know that the nice case happens infinitely often;
    therefore,
    \begin{equation*}
        b_j = \limsup_{\substack{
            k \to \infty \\
            \text{$V_k$ not-so-good}
        }}
        \frac{1}{s'_k - s_k} - k \langle a_j, w_0 \rangle.
    \end{equation*}

    Therefore,
    if $F_j$ is a facet of $P$,
    both when there are infinitely many good $V_k$
    and when there are infinitely many not-so-good $V_k$
    we can compute $\rvol F_j$ and $b_j$.

    Now apply induction on $j$ to finish the proof.
\end{proof}

\begin{corollary}
    \label{thm:semirational-complete-invariant}
    Let $P$ and $Q$ be two full-dimensional semi-rational polytopes
    such that $L_{P + w}(s) = L_{Q + w}(s)$
    for all integer $w$
    and all real $s > 0$.
    Then $P = Q$.
\end{corollary}

In other words,
the functions $L_{P + w}(s)$ for all integer $w$
form a complete set of invariants
in the class of full-dimensional semi-rational polytopes.

\begin{proof}
    If $a$ is a primitive integer vector such that
    \begin{equation*}
        P \cap \{x \in \mathbb R^d \mid \langle a, x \rangle = b \}
    \end{equation*}
    is a facet of $P$,
    then,
    as $Q$ is a bounded set,
    for some appropriate $b'$ we have
    \begin{equation*}
        Q \subseteq \{x \in \mathbb R^d \mid \langle a, x \rangle \leq b' \}.
    \end{equation*}

    This means we can write
    \begin{align*}
        P &= \bigcap_{i=1}^n \{
                x \in \mathbb R^d \mid \langle a_i, x \rangle \leq b_i
            \} \\
        Q &= \bigcap_{i=1}^n \{
                x \in \mathbb R^d \mid \langle a_i, x \rangle \leq c_i
            \},
    \end{align*}
    where $a_1, \dots, a_n$ are primitive integer vectors,
    and $b_1, \dots, b_n, c_1, \dots, c_n$ are real numbers.

    That is,
    by possibly adding some redundant vectors,
    we can write $P$ and $Q$ as intersection of hyperplanes
    using the same set of normal vectors.
    Now just apply Theorem~\ref{thm:recovering-right-hand-sides}
    (assuming knowledge of the vectors $a_1, \dots, a_n$)
    to conclude that $P = Q$.
\end{proof}

%% file: codimension-one.tex
\section{Codimension one polytopes}
\label{sec:codimension-one}

Corollary~\ref{thm:semirational-complete-invariant}
says that the functions $L_{P + w}(s)$ form a complete set of invariants
in the class of full-dimensional semi-rational polytopes.
There exists a simple example which shows that
full-dimensionality is needed:
consider again the affine space $M$
defined in Section~\ref{sec:translation-variant} by
\begin{equation*}
    M = \{(\ln 2, \ln 3)\} \times \mathbb R^{d-2}.
\end{equation*}

For any real $s > 0$ and any integer $w$,
the affine space $s(M + w)$ has no integer points,
and thus if $P$ and $Q$ are any polytopes which are contained in $M$
then $L_{P + w}(s) = L_{Q + w}(s) = 0$
for all $w$ and all $s > 0$.

$P$ and $Q$ may be chosen to be semi-rational in the example above,
so we know that
the analogue of Corollary~\ref{thm:semirational-complete-invariant}
for codimension $2$ semi-rational polytopes
is false.
This leaves open the possibility that
the analogue for semi-rational codimension $1$ polytopes,
or for rational polytopes of any dimension,
is true.
In this section
we will show that both these analogues are indeed true.

The general idea is to reduce to the full-dimensional case,
but we will need to do some adjustments.
For example,
if $P$ is a $(d-1)$-dimensional polytope
contained in $\mathbb R^{d-1} \times \{\frac 12\}$,
let $P' = P - (0, \dots, 0, \frac 12)$,
and $P''$ be the projection of $P'$ to $\mathbb R^{d-1}$.
The polytope $P''$ is,
indeed,
a full-dimensional polytope in $\mathbb R^{d-1}$,
and if $w = (w_1, \dots, w_{d-1})$ is an integer vector
we know that
\begin{equation*}
    L_{P' + w}(s) = L_{P'' + (w_1, \dots, w_{d-1}, 0)}(s)
\end{equation*}
for all $s$.
Therefore,
if we can compute $L_{P' + w}(s)$ for all $w$ whose last coordinate is zero,
we may use Corollary~\ref{thm:semirational-complete-invariant}
for $P''$,
and conclude $P''$ is uniquely identified.
(We'll see later how to distinguish between two translates of the same polytope.)

The problem is that,
just by using $L_{P + w}(s)$,
we cannot compute $L_{P' + w}(s)$ for all $s$.
Let $w = (w_1, \dots, w_d)$ be an integer vector.
Then
\begin{equation*}
    P + w \subseteq \mathbb R^{d-1} \times \{w_d + \tfrac 12\},
\end{equation*}
so $L_{P + w}(s)$ will be nonzero only for $s$ of the form $\frac{m}{w_d + \frac12}$
for some integer $m$.
In this case,
we have
\begin{equation*}
    L_{P + w}(s) = L_{P' + w'}(s),
\end{equation*}
for $w' = (w_1, \dots, w_{d-1}, 0)$.

Thus,
we may compute $L_{P'' + w}(s)$
only for $s$ of the form $\frac{2m}{2 w_d + 1}$;
that is,
instead of knowing the value of $L_{P'' + w}(s)$ for all $s$,
we know it just for a dense subset of $\mathbb R$.

Since each $L_P(s)$ is piecewise constant,
this is still enough information to compute the one-sided limits
$L_P(s^+)$ and $L_P(s^-)$
for all $s > 0$,
and as each $L_P(s)$ is lower semicontinuous,
we can fully reconstruct most of its discontinuities.
In particular,
if $s_0$ is either a right-discontinuity or a left-discontinuity,
but no both,
of $L_P(s)$,
we know that $L_P(s_0)$ is the largest of $L_P(s_0^-)$ and $L_P(s_0^+)$.
However,
this is not enough to recover $L_P(s_0)$
if $s_0$ is both a left- and right-discontinuity;
this happens,
for example,
at $s_0 = 3$ with the square of Figure~\ref{fig:ehrhart-function-small-square}.

In order to use Corollary~\ref{thm:semirational-complete-invariant},
we will strengthen its proof
to work with knowledge of $L_{P + w}(s)$ only for densely many $s$.
More specifically,
we will modify Lemma~\ref{thm:isolating-largest-vector}
so that the choice of $w_0$ avoids these overlapping discontinuities,
at least in the window $(\alpha_k, \alpha_k + \epsilon_k)$
in which we analyze $L_{P + kw_0}(s)$.

%% file: avoiding-overlapping-discontinuities.tex
\subsection{Avoiding overlapping discontinuities}

Again,
write $P$ as
\begin{equation*}
    P = \bigcap_{i=1}^n \{ x \in \mathbb R^d \mid \langle a_i, x \rangle \leq b_i \},
\end{equation*}
and let $F_i$ be the corresponding facets of $P$.

The facet $F_i + w_k$ may only trigger a discontinuity at $s$ if
\begin{equation*}
    s = \frac{m}{b_i + \langle a_i, w_k \rangle}
\end{equation*}
for some integer $m$.
Therefore,
if $F_i + w_k$ and $F_j + w_k$ both trigger a discontinuity at $s$,
then we must have,
for integers $m_i$, $m_j$,
\begin{equation*}
    s = \frac{m_i}{b_i + \langle a_i, w_k \rangle}
      = \frac{m_j}{b_j + \langle a_j, w_k \rangle}.
\end{equation*}

If we divide both $m_i$ and $m_j$ by their greatest common divisor,
we obtain a ``primitive'' number $s$
such that all other simultaneous discontinuity points of $F_i + w_k$ and $F_j + w_k$
are integer multiples of $s$.
So,
assume $m_i$ and $m_j$ are relatively prime.

From this equation,
it is clear that if both $F_i + w_k$ and $F_j + w_k$ trigger a discontinuity,
we must either have both $b_i$ and $b_j$ rational,
or both irrational.
We will first deal with the irrational case,
which is easier.

Assume for now that $b_i$ and $b_j$ are both irrational.
We may rewrite the above equation as
\begin{equation*}
    b_j = \frac{m_i}{m_j} b_i +
        \frac{m_i}{m_j} \langle a_j, w_k \rangle - \langle a_i, w_k \rangle.
\end{equation*}

If there is a $k' \neq k$
for which $F_i + w_{k'}$ and $F_j + w_{k'}$ also have overlapping discontinuities,
then there is also two coprime integers $m'_i, m'_j$ such that
\begin{equation*}
    b_j = \frac{m'_i}{m'_j} b_i +
        \frac{m'_i}{m'_j} \langle a_j, w_{k'} \rangle - \langle a_i, w_{k'} \rangle.
\end{equation*}

In both cases,
we wrote $b_j$ as a rational combination of $b_i$ and $1$.
Thinking of $\mathbb R$ as a vector space over $\mathbb Q$,
we know $b_i$ and $1$ are linearly independent,
and thus this representation of $b_j$ is unique.
Therefore,
\begin{equation*}
    \frac{m'_i}{m'_j} = \frac{m_i}{m_j}
    \qquad \text{and} \qquad
    \frac{m_i}{m_j} \langle a_j, w_k \rangle - \langle a_i, w_k \rangle =
    \frac{m'_i}{m'_j} \langle a_j, w_{k'} \rangle - \langle a_i, w_{k'} \rangle.
\end{equation*}

Since the pairs $(m_i, m_j)$ and $(m'_i, m'_j)$ are of coprime numbers,
the first of these two equations give
$m_i = \pm m'_i$ and $m_j = \pm m'_j$;
we may assume $m_i = m'_i$ and $m_j = m'_j$.
Using these identities
and expanding $w_k = k w_0$,
the second equation may be rearranged to
\begin{equation*}
    k \langle m_i a_j - m_j a_i, w_0 \rangle = k' \langle m_i a_j - m_j a_i, w_0 \rangle.
\end{equation*}

Since we assumed that $k' \neq k$,
we must have
\begin{equation*}
    \langle m_i a_j - m_j a_i, w_0 \rangle = 0,
\end{equation*}
which means that $w_0$ is orthogonal to $m_i a_j - m_j a_i$.

Therefore,
if we guarantee that $w_0$ is not orthogonal to $m_i a_j - m_j a_i$,
we are sure that $F_i + w_k$ and $F_j + w_k$ will have overlapping discontinuities
for at most one $k$.

Ideally,
we would like to make $w_0$ non-orthogonal to $m_i a_j - m_j a_i$
for all possible choices of $m_i, m_j$ and all $a_i, a_j$.
That would add an infinite number of restrictions on $w_0$,
which might make the choice of $w_0$ impossible
(for example, if $a_1 = (1, 0)$ and $a_2 = (0, 1)$,
then the only possible choice would be $w_0 = 0$).

Fortunately,
there is at most one problematic $m_i$ and $m_j$ which must be avoided
for each pair of irrationals $b_i$ and $b_j$.
So,
for convenience,
call the ``dependence index''
between any two irrational numbers $b_i$ and $b_j$
to be $\max\{|m_i|, |m_j|\}$,
where $m_i$ and $m_j$ are coprime numbers such that $b_i = \frac{m_i}{m_j} b_j + r$
for some rational number $r$.
(This is well-defined because if $m_i$ and $m_j$ exist,
then they are unique,
up to signs.)
If no such integers $m_i$ and $m_j$ exist,
let the dependence index be zero.

If $N$ is larger than the dependence index
between any two irrational numbers in $\{b_1, \dots, b_n\}$,
making $w_0$ not orthogonal to any vector of the form
$m_i a_j - m_j a_i$ with $|m_i|, |m_j| \leq N$ guarantees that
the discontinuities of $F_i + w_k$ and $F_j + w_k$
will overlap for at most one $k$.

This solves the irrational clashes,
so now assume $b_i$ and $b_j$ are rational.

Here,
clashes are unavoidable.
The goal is to make them happen outside $(\alpha_k, \alpha_k + \epsilon_k)$.
This will be accomplished by showing that,
if $s$ is a simultaneous discontinuity point for any $k$,
then the denominator of $s$ is bounded.
The interval $(\alpha_k, \alpha_k + \epsilon_k)$
contains points of the form $k + \delta$
for small, positive $\delta$,
and as $k$ gets large,
$\delta$ gets small.
This will guarantee no discontinuity clashes happen inside this interval.

Write $b_i = \frac{p_i}{q_i}$.
The ``primitive clash equation'' reads
\begin{equation*}
    s = \frac{q_i m_i}{p_i + q_i \langle a_i, w_k \rangle}
      = \frac{q_j m_j}{p_j + q_j \langle a_j, w_k \rangle},
\end{equation*}
which may be rewritten as
\begin{equation*}
    m_i q_i (p_j + q_j \langle a_j, w_0 \rangle k) =
    m_j q_j (p_i + q_i \langle a_i, w_0 \rangle k).
\end{equation*}

This is an equation of the form $\zeta m_i = \xi m_j$,
where $\zeta = q_i p_j + q_i q_j \langle a_j, w_0 \rangle k$
and $\xi = q_j p_i + q_j q_i \langle a_i, w_0 \rangle k$.
All solutions are of the form
\begin{equation*}
    (m_i, m_j) = \left( \frac{\xi}{\gcd(\zeta, \xi)} l,
                        \frac{\zeta}{\gcd(\zeta, \xi)} l \right),
\end{equation*}
for some integer $l$.
Since we are looking for the ``primitive'' solution,
we'll take $l = 1$.

The following lemma guarantees that
the denominator of the primitive solution is bounded,
as a function of $k$.

\begin{lemma}
    \label{thm:gcd-arithmetic-sequence}
    Let $\zeta, \xi, \gamma, \eta$ be integers
    such that the vectors $(\zeta, \gamma)$ and $(\xi, \eta)$
    are linearly independent.
    Then the sequence $\{\gcd(\zeta k + \gamma, \xi k + \eta)\}_{k = 1}^{\infty}$
    is periodic.
\end{lemma}

\begin{proof}
    By swapping $(\zeta, \gamma)$ and $(\xi, \eta)$
    and multiplying them by $-1$,
    if needed,
    we may assume that $\zeta \geq \xi \geq 0$.

    If $\xi = 0$,
    then the sequence has only terms of the form
    \begin{equation*}
        \gcd( \zeta k + \gamma, \eta ).
    \end{equation*}

    Since $\eta \neq 0$
    (due to the linear independence restriction),
    we have
    \begin{equation*}
        \gcd( \zeta k + \gamma, \eta ) =
        \gcd(\eta, (\zeta k + \gamma) \bmod \eta),
    \end{equation*}
    from which periodicity is clear.

    If $\zeta, \xi > 0$,
    we have
    \begin{equation*}
        \gcd(\zeta k + \gamma, \xi k + \eta) =
        \gcd(\xi k + \eta, (\zeta - \xi)k + \gamma - \eta),
    \end{equation*}
    so if we let $\zeta' = \xi$,
    $\xi' = \zeta - \xi$,
    $\gamma' = \eta$,
    and $\eta' = \gamma - \eta$,
    then $\zeta', \xi' \geq 0$ and $\zeta' + \xi' < \zeta + \xi$,
    so we may apply induction in $\zeta + \xi$
    to conclude that the sequence is periodic.
\end{proof}

In our case,
we have $\zeta = q_i q_j \langle a_i, w_0 \rangle$,
$\xi = q_i q_j \langle a_i, w_0 \rangle$,
$\gamma = q_j p_i$ and $\eta = q_i p_j$.
For the vectors $(\zeta, \gamma)$ and $(\xi, \eta)$ to be linearly dependent,
there must exist some rational number $r$ such that
$r\zeta = \xi$ and $r \gamma = \eta$.
That is,
\begin{equation*}
    \frac{q_i q_j \langle a_i, w_0 \rangle} {q_i q_j \langle a_j, w_0 \rangle}
    = \frac \xi \zeta = r = \frac \eta \gamma =
    \frac{q_i p_j}{q_j p_i}
\end{equation*}

Rearranging the equation gives
\begin{equation*}
    \langle w_0, p_j q_i a_i - p_i q_j a_j \rangle = 0.
\end{equation*}
Therefore,
if we can guarantee that $w_0$ is not orthogonal to
$p_j q_i a_j - p_i q_j a_j$,
for all $i$ and $j$
(with $i \neq j$),
we will make sure that
\begin{equation*}
    \frac{m_i}{p_i + q_i \langle a_i, w_k \rangle}
\end{equation*}
will have only finitely many distinct values.
As $s$ is a multiple of this value,
this gives a bound on the denominator of $s$;
so,
for large enough $k$,
no such $s$ will appear in $(\alpha_k, \alpha_k + \epsilon_k)$.

We may use the same trick that deals with the irrational case:
we will add a parameter $N$ to the lemma,
and make $w_0$ orthogonal to all $\tau_j a_j - \tau_i a_i$
for any pair $\tau_i, \tau_j$ with $|\tau_i|, |\tau_j| \leq N^2$.

The modified lemma is the following.

\begin{lemma}
    \label{thm:densely-isolating-largest-vector}
    Let $N \geq 0$ be some integer number,
    and let $a_1, \dots, a_n$ be primitive integer vectors in $\mathbb R^d$,
    with $\norm{a_1} \geq \norm{a_i}$ for all $i$.
    Then there is an integer vector $w_0$
    and a sequence $(\alpha_k, \alpha_k + \epsilon_k)$ of intervals
    such that,
    for all possible choice of real numbers $b_1, \dots, b_n$,
    all the properties enumerated in Lemma~\ref{thm:isolating-largest-vector}
    are true,
    and also the following one:
    \begin{enumerate}[technical-lemma]
        \item \label{item:no-discontinuity-clash}
            If the numerator and the denominator of all rational $b_i$
            is smaller than or equal to $N$,
            and the ``dependence index'' between any two irrational $b_i$
            is smaller than or equal to $N$,
            then for all sufficiently large $k$
            there will be no $s \in (\alpha_k, \alpha_k + \epsilon_k)$
    \end{enumerate}
\end{lemma}

\begin{proof}
    In the choice of the vector $w_0$,
    in the beginning of the lemma,
    there was already a (finite) list of integer vectors
    such that $w_0$ was made non-orthogonal
    (namely, all vectors $a_i$ for $i = 1, \dots, n$).
    Now,
    add to that list all vectors of the form
    \begin{equation*}
        \tau_i a_j - \tau_j a_i,
    \end{equation*}
    for all pairs of integers $\tau_i, \tau_j$ such that $|\tau_i|, |\tau_j| \leq N^2$.

    Primitiveness of the vectors $a_i$ guarantee that
    none of these vectors is the zero vector,
    and as we are limiting the coefficients by $N^2$
    we still have a finite list.

    Now choose $w_0$, $\alpha_k$ and $\epsilon_k$
    in the same way as in the proof of Lemma~\ref{thm:isolating-largest-vector}
    except that, now,
    the list of vectors not orthogonal to $w_0$ is larger.
    This will guarantee all the properties of that lemma.

    Finally,
    property~\ref{item:no-discontinuity-clash}
    follows from the discussion above.
\end{proof}

%% file: piecing-together-avoiding-discontinuities.tex
\subsection{Piecing together the semi-rational case, but avoiding discontinuity clashes}
\label{sec:piecing-together-avoiding-discontinuities}

The improved Lemma~\ref{thm:densely-isolating-largest-vector}
guarantees that,
for large enough $k$,
there will not be discontinuities overlapping in $(\alpha_k, \alpha_k + \epsilon_k)$,
and thus the other lemmas may be used.
So we can prove the following improved version
of Theorem~\ref{thm:recovering-right-hand-sides}.

\begin{theorem}
    \label{thm:recovering-right-hand-sides-dense-information}
    Let $a_1, \dots, a_n \in \mathbb R^d$ be primitive integer vectors
    and for each integer $w$ let $f_w(s)$ be a function on $\mathbb R$.
    Then there is at most one set of numbers $b_1, \dots, b_n$ such that
    \begin{equation*}
        P = \bigcap_{i=1}^n \{x \in \mathbb R^d \mid \langle a_i, x \rangle \leq b_i\}
    \end{equation*}
    is a full-dimensional semi-rational polytope,
    that $f_w(s) = L_{P + w}(s)$ for all integer $w$
    and all $s$ in a dense subset of $\mathbb R$,
    and that the polytopes $F_i$ defined by
    \begin{equation*}
        F_i = P \cap \{x \in \mathbb R^d \mid \langle a_i, x \rangle = b_i\}
    \end{equation*}
    are faces of $P$.
\end{theorem}

\begin{proof}
    Write $b_i = \frac{p_i}{q_i}$ whenever $b_i$ is rational,
    and let $N$ be so large that $N \geq |p_i|$ and $N \geq |q_i|$,
    and that $N$ is larger than the dependence index
    between any two irrational numbers in $\{b_1, \dots, b_n\}$.
    Use Lemma~\ref{thm:densely-isolating-largest-vector} with this $N$,
    and property~\ref{item:no-discontinuity-clash}
    will guarantee that,
    for all large enough $k$,
    we still know precisely where the discontinuities of $L_{P + w}(s)$ happen
    and what are their magnitudes.

    Thus the remainder of the proof is identical.
\end{proof}

\begin{corollary}
    \label{thm:semirational-dense-information-complete-invariant}
    Let $P$ and $Q$ be two full-dimensional semi-rational polytopes
    such that $L_{P + w}(s) = L_{Q + w}(s)$
    for all integer $w$
    and all real $s > 0$ in a dense subset of $\mathbb R$.
    Then $P = Q$.
\end{corollary}

\begin{proof}
    Analogous to the proof of \ref{thm:semirational-complete-invariant}.
\end{proof}

%% file: semirational-codimension-one-polytopes.tex
\subsection{Semi-rational polytopes with codimension $0$ and $1$}

Now we may show Corollary~\ref{thm:semirational-complete-invariant}
for semi-rational polytopes
which have codimension $0$ and $1$.

\begin{restatetheorem}{thm:codimension-zero-and-one}
    Let $P$ and $Q$ be two semi-rational polytopes in $\mathbb R^d$,
    both having codimension $0$ or $1$.
    Suppose moreover that $L_{P + w}(s) = L_{Q + w}(s)$
    for all integer $w$ and all real $s > 0$.
    Then $P = Q$.
\end{restatetheorem}

\begin{proof}
    If their codimensions do not match,
    then $L_P(s)$ and $L_Q(s)$ will be different,
    so we may assume either both have codimension $0$
    or both have codimension $1$.
    In the first case,
    $P$ and $Q$ are full-dimensional,
    so we may use Corollary~\ref{thm:semirational-dense-information-complete-invariant}
    directly.
    Thus,
    assume both $P$ and $Q$ have codimension $1$.

    Let $H$, $H'$ satisfy
    \begin{align*}
        P \subseteq H &= \{x \in \mathbb R^d \mid \langle a, x \rangle = b\} \\
        Q \subseteq H' &= \{x \in \mathbb R^d \mid \langle a', x \rangle = b'\}.
    \end{align*}

    If we had $a \neq a'$,
    we could choose some vector $w$ which is orthogonal to $a$ but not to $a'$,
    and then $\vol \ppyr(Q + kw)$ would increase for large $k$
    (because the height of the pseudopyramid would increase,
    whereas the area of the base do not change)
    but $\vol \ppyr(P + kw)$ would stay the same
    (because neither the height nor the area of the base would change).
    Since $\vol \ppyr(P + kw)$ is determined by $L_{P + kw}(s)$
    (by Lemma~\ref{thm:different-pseudopyramid-volumes}),
    we know this cannot happen.

    This shows $a = a'$,
    and using Lemma~\ref{thm:limit-is-relative-volume}
    both $L_P(s)$ and of $L_Q(s)$ must exhibit discontinuities
    for all large enough $s$,
    which shows $b = b'$.
    Thus, $H = H'$.

    Now using unimodular transforms we may assume that
    \begin{equation*}
        H = \{x \in \mathbb R^d \mid x_d = b \}.
    \end{equation*}

    We have $P - (0, \dots, 0, b) \in \mathbb R^{d-1} \times \{0\}$,
    and analogously for $Q$,
    so we may define $P' \subseteq \mathbb R^{d-1}$
    to be the projection of $P - (0, \dots, 0, b)$ to $\mathbb R^{d-1}$,
    and analogously for $Q$.
    We will show that $P' = Q'$,
    which implies $P = Q$.

    Let $w' = (w_1, \dots, w_{d-1})$ be given.
    If $s$ is of the form
    \begin{equation*}
        s = \frac{m}{b + w_d}
    \end{equation*}
    for some integers $m$ and $w_d$,
    let $w = (w_1, \dots, w_d)$ and then
    \begin{align*}
        L_{P + w}(s) &= \#( s(P + w) \cap \mathbb Z^d ) \\
            &= \#\big( (s(P + w) - (0, \dots, 0, m)) \cap \mathbb Z^d \big) \\
            &= \#(s(P' + w') \cap \mathbb Z^{d-1}) \\
            &= L_{P' + w'}(s).
    \end{align*}

    Analogously, we have $L_{Q + w}(s) = L_{Q' + w'}(s)$.

    This shows that $L_{P' + w'}(s) = L_{Q' + w'}(s)$
    for all integer $w' \in \mathbb R^{d-1}$
    and all $s > 0$ of the form $\frac{m}{b + w_d}$,
    which form a dense subset of $\mathbb R$.
    Therefore,
    $P'$ and $Q'$ satisfy the hypothesis
    of Corollary~\ref{thm:semirational-dense-information-complete-invariant},
    and thus $P' = Q'$,
    which shows $P = Q$.
\end{proof}

%% file: all-the-way-down.tex
\subsection{All the way down with the rationals}

As a last bonus
from Corollary~\ref{thm:semirational-dense-information-complete-invariant},
we may extend Theorem~\ref{thm:codimension-zero-and-one}
for all dimensions,
if we restrict ourselves to rational polytopes.

\begin{restatetheorem}{thm:rational-complete-invariant}
    Let $P$ and $Q$ be two rational polytopes in $\mathbb R^d$.
    Suppose that $L_{P + w}(s) = L_{Q + w}(s)$
    for all integer $w$ and all real $s > 0$.
    Then $P = Q$.
\end{restatetheorem}

\begin{proof}
    Measuring rate of growth of $L_P(s)$ and $L_Q(s)$
    gives the dimension of both polytopes,
    so we may assume both have the same dimension.
    If they have codimension $0$ or $1$,
    then we may apply Theorem~\ref{thm:codimension-zero-and-one} directly.
    So,
    assume their dimension is smaller than $d-2$.

    Let $M = \aff P$,
    and let $H$ be the translate of $M$ which passes through the origin.
    Define $M'$ and $H'$ analogously for $Q$.
    If $H \neq H'$,
    let $w \in H \setminus H'$;
    then the relative volume of $\ppyr(P + kw)$ will stay the same,
    whereas the relative volume of $\ppyr(Q + kw)$ will increase for large $k$
    (we may measure the relative volume for $P + kw$ and $Q + kw$
    because these polytopes are rational,
    and thus there are dilations of them which will contain integer points.)

    This shows $H = H'$,
    and a similar reasoning as before shows $M = M'$.
    Now let $A$ be a unimodular transform
    which maps $M$ to a subset of $\mathbb R^{d-1} \times \{0\}$;
    apply $A$ to both $P$ and $Q$,
    project the unimodular images to $\mathbb R^{d-1}$,
    and use this theorem for $d-1$ to conclude $P = Q$.
\end{proof}

%% file: final-remarks.tex
\section{Final Remarks}

Theorem~\ref{thm:codimension-zero-and-one}
assumes knowledge of $L_{P + w}(s)$ for all integer $w$
and guarantees that $P$ is uniquely determined,
as long as $P$ is a full-dimensional semirational polytope.
We mention two open questions regarding this theorem.

The first question
(already hinted in Section~\ref{sec:non-rational-polytopes})
whether we may extend Theorem~\ref{thm:codimension-zero-and-one}
for all real polytopes.
We have the following conjecture.

\begin{conjecture}
    Let $P$ and $Q$ be two full-dimensional real polytopes
    such that $L_{P + w}(s) = L_{Q + w}(s)$
    for all integer $w$ and all real $s > 0$.
    Then $P = Q$.
\end{conjecture}

The second question is whether we need to know $L_{P + w}$ for \emph{all} $w$.
For the sake of naming,
let us call a certain set $W \subseteq \mathbb Z^d$ a \emph{witness set}
if $L_{P + w}(s) = L_{Q + w}(s)$ for all $w \in W$ implies $P = Q$;
for example,
Theorem~\ref{thm:rational-complete-invariant}
says that $W = \mathbb Z^d$ is a witness set for the class of rational polytopes.

\begin{question}
    \label{q:sinai}
    Does there exist a finite witness set
    (for example, for the class of rational polytopes)?
\end{question}

\subsection{Acknowledgements}
The author wants to thank Sinai Robins for reading an early version of this text
and suggesting Question~\ref{q:sinai}.